\documentclass[twoside, 12pt]{article}
\usepackage{a4wide, amsthm, amssymb, amscd, mathptmx}

\usepackage{amsfonts}
\usepackage{cite}
\usepackage[all]{xy}\CompileMatrices\SelectTips{cm}{12}

\theoremstyle{plain}
\newtheorem{Thm}{\sc Theorem}[section]
\newtheorem{Theorem}[Thm]{\sc Theorem}
\newtheorem{Corollary}[Thm]{\sc Corollary}
\newtheorem*{Corollary*}{\sc Corollary}
\newtheorem{Proposition}[Thm]{\sc Proposition}
\newtheorem*{Proposition*}{\sc Proposition}
\newtheorem{Lemma}[Thm]{\sc Lemma}

\theoremstyle{definition}
\newtheorem{Definition}[Thm]{Definition}

\theoremstyle{remark}
\newtheorem{Remark}[Thm]{Remark}
\newtheorem{Example}[Thm]{Example}
\newtheorem*{Example*}{Example}
\newtheorem*{Remark*}{Remark}

\newcommand{\Ad}{\mathop{\rm Ad}}
\newcommand{\id}{{\mathop{\rm id}}}
\newcommand{\ZZ}{{\mathbb Z}}
\newcommand{\Spec}{\mathop{\rm Spec}}
\renewcommand{\O}{{\cal O}}
\newcommand{\E}{{\cal E}}
\newcommand{\F}{{\cal F}}
\newcommand{\G}{{\cal G}}
\newcommand{\fg}{{\mathfrak g}}
\newcommand{\NN}{{\mathbb N}}
\newcommand{\X}{{\cal X}}
\newcommand{\PP}{{\mathbb P}}
\newcommand{\QQ}{{\mathbb Q}}
\newcommand{\CC}{{\mathbb C}}
\newcommand{\C}{{\cal C}}
\newcommand{\ch}{\mathop{\rm ch}}
\newcommand{\chr}{\mathop{\rm char}}
\newcommand{\td}{\mathop{\rm td}}
\newcommand{\cS}{{\cal S}}
\newcommand{\cT}{{\cal T}}
\newcommand{\ti}{\tilde}
\newcommand{\Hom}{{\mathop{{\rm Hom}}}}
\newcommand{\cHom}{{\mathop{{\cal H}om}}}
\newcommand{\End}{{\mathop{{\cal E}nd}}}
\newcommand{\Vect}{{\mathop{Vect} ^s_0}}
\newcommand{\Ext}{{\mathop{{\rm E}xt}}}
\newcommand{\cExt}{{\mathop{{\cal E}xt}}}
\newcommand{\ext}{{\mathop{{\rm ext}}}}
\newcommand{\ob}{\mathop{{ob}}}
\newcommand{\Num}{\mathop{\rm Num}}
\newcommand{\num}{\mathop{\rm num}}
\newcommand{\im}{\mathop{\rm im}}
\newcommand{\coker}{\mathop{\rm coker}}
\newcommand{\mod}{{\mathop{{\rm mod}}}}
\newcommand{\GL}{\mathop{\rm GL}}
\newcommand{\Pic}{\mathop{\rm Pic}}

\begin{document}

\markboth{\rm A.\ Langer} {\rm On the S-fundamental group scheme}

\title{On the S-fundamental group scheme}
\author{Adrian Langer}
\date{\today}

\maketitle


{\sc Address:}\\
1. Institute of Mathematics, Warsaw University,
ul.\ Banacha 2, 02-097 Warszawa, Poland,\\
2. Institute of Mathematics, Polish Academy of Sciences,
ul.~\'Sniadeckich 8, 00-956 Warszawa, Poland.\\
e-mail: {\tt alan@mimuw.edu.pl}

\begin{abstract} We introduce  a new fundamental group scheme
for varieties defined over an algebraically closed (or just perfect)
field of positive characteristic and we use it to study generalization
of C. Simpson's results to positive characteristic. We also
study the properties of this group and we prove Lefschetz type
theorems.
\end{abstract}

\section*{Introduction}

A. Grothendieck as a substitute of a topological fundamental group
introduced the \'etale fundamental group, which in the complex
case is just a profinite completion of the topological fundamental
group. The definition uses all finite \'etale covers and in
positive characteristic it does not take into account inseparable
covers. To remedy the situation M. Nori introduced the fundamental
group scheme which takes into account all principal bundles with
finite group scheme structure group. In characteristic zero this
recovers the \'etale fundamental group but in general it carries
more information about the topology of the manifold. Obviously,
over complex numbers the topological fundamental group carries
much more information than the \'etale fundamental group. To
improve this situation  C. Simpson introduced in \cite{Si} the
{universal complex pro-algebraic group} (or an algebraic envelope
of  the topological fundamental group in the language of
\cite[10.24]{De}). This group carries all the information about
finite dimensional representations of the topological fundamental
group. On this group Simpson introduced  a non-abelian Hodge
structure which gives rise to a non-abelian Hodge theory.

The main aim of this paper is to generalize some of his results to
positive characteristic. As a first step to this kind of
non-abelian Hodge theory we study the quotient of the universal
complex pro-algebraic group which, in the complex case,
corresponds to the Tannakian category of  holomorphic flat bundles
that are extensions of unitary flat bundles. Via the well known
correspondence started with the work of M. S. Narasimhan and C. S.
Seshadri, objects in this category correspond to semistable vector
bundles with vanishing Chern classes.

In positive characteristic we take this as a starting point of our
theory. In particular, in analogy to \cite[Theorem 2]{Si} we prove
that strongly semistable sheaves with vanishing Chern classes are
locally free. We use this to prove that such sheaves are
numerically flat (i.e., such nef locally free sheaves whose dual
is also nef). We also prove the converse: all numerically flat
sheaves are strongly semistable and they have vanishing Chern
classes (in complex case this equivalence follows from
\cite[Theorem 1.18]{DPS}).

This motivates our definition of the S-fundamental group scheme
(see Definition \ref{Def}). Namely, we define the S-fundamental
group scheme as Tannaka dual to the neutral Tannaka category of
numerically flat sheaves. Note that in this definition we do not
need neither smoothness nor projectivity of the variety for which
we define the S-fundamental group scheme.

However, considering reflexive sheaves with vanishing Chern
classes on smooth projective varieties is sometimes much more
useful. For example, notion of strong stability can be used to
formulate some interesting restriction theorems (see Section 4)
that are used in proofs of Lefschetz type theorems. It is also of
crucial importance in several other proofs.

The S-fundamental group scheme always allows to recover Nori's
fundamental group scheme. In fact, Nori in \cite{No} considered a
closely related category of degree $0$ vector bundles whose
pull-backs by birational maps from smooth curves are semistable.
Recently, the S-fundamental scheme group was defined in the curve
case in \cite[Definition 5.1]{BPS} (in this case there are no
problems caused, e.g., by non-locally free sheaves).

If  the cotangent sheaf of the variety does not contain any
subsheaves of non-negative slope (with respect to some fixed
polarization) then in the complex case the S-fundamental group
scheme is equal to  Simpson's universal complex pro-algebraic
group (note that the corresponding non-abelian Hodge structure is
in this case trivial). In positive characteristic, under the same
assumption, we prove that the S-fundamental group scheme allows us
to recover all known fundamental groups like Deligne-Shiho's
pro-unipotent completion of the fundamental group or dos Santos'
fundamental group scheme obtained by using all $\O_X$-coherent
${\cal D}_X$-modules (or stratified sheaves). Note that in this
case we also get projective (!) moduli space structure on the
non-abelian cohomology set $H^1(\pi_1^S(X,x), \GL _k(n))$,
corresponding to the Dolbeaut moduli space (this follows from
Theorem \ref{loc-free}).

A large part of the paper is devoted to study the properties of
the S-fundamental group scheme. It satisfies the same properties
as Nori's fundamental group scheme. Many of the properties are
quite easy to prove but some as in the case of Nori's fundamental
group scheme are quite difficult. For example, the behavior under
tensor products for Nori's fundamental group scheme was studied
only in \cite{MS}. The corresponding result for the S-fundamental
group scheme uses completely different techniques and it is
subject of the second part of this paper.

One of the main results of this paper are Lefschetz type theorems
for the S-fundamental group scheme. As a corollary get the
corresponding results for Nori's (and \'etale) fundamental groups.
This corollary was proved in \cite{BH} in a much more cumbersome
way using Grothendieck's Lefschetz theorems for the \'etale
fundamental group. Our proofs are quite quick and they depend on
some vanishing of cohomology proven using the techniques described
by Szpiro in \cite{Szp}.

Our proof of the Lefschetz type theorems for the S-fundamental
group scheme is quite delicate as we need to extend vector bundles
from ample divisors and this usually involves vanishing of
cohomology that even in characteristic zero we cannot hope for
(see the last part of Section \ref{Section-p^2}). A similar
problem occured in Grothendieck's proof of Lefschetz theorems for
Picard groups. In this case the Picard scheme of a smooth surface
in $\PP ^3$ is not isomorphic to $\ZZ$ (for example for a cubic
surface) and Lefschetz theorem for complete intersection surfaces
says that the component of the numerically trivial divisors in the
Picard scheme is trivial (see \cite[Expose XI, Th\`eor\'eme
1.8]{DK}). Our theorem gives information about the Picard scheme
not only in case of hypersurfaces in projective spaces but for
ample divisors in arbitrary projective varieties (also if the
Picard scheme of the ambient variety is non-reduced). One just
needs to notice that the component of the numerically trivial
divisors in the Picard scheme is equal to the group of characters
of the S-fundamental group scheme.

In the higher rank case there also appears another problem:
extension of a vector bundle on a divisor need not be a
vector bundle. This is taken care of by Theorem \ref{loc-free}
(which partially explains why we bother with semistable sheaves
and not just numerically flat vector bundles).

In the last section we use the lemma of Deligne and Illusie to
give a quick proof of Lefschetz type theorems for the
S-fundamental group scheme for varieties which admit a lifting
modulo $p^2$.

We should note that a strong version of boundedness of semistable
sheaves (see \cite{La1} and \cite{La2}) is frequently used in
proofs in this paper (although we could do without it in many but
not all places).

\medskip
To prevent the paper to grow out of a reasonable size we decided to
skip many interesting topics. In future we plan to treat
the (full) universal pro-algebraic fundamental group and a tame version of this
group for non-proper varieties. We also plan to add some applications
to the study of varieties with nef tangent bundle (for this purpose the
results of this paper are already sufficient).

\medskip

The structure of the paper is as follows. In Section 1 we recall a
few well known results. In particular, Subsection 1.3 motivates
the results of Section 4. In Section 2 we recall some boundedness
results used in later proofs. We also use them to prove some
results on deep Frobenius descent generalizing H. Brenner's and A.
Kaid's results \cite{HL}. These results are are of independent
interest and they are not used later in the paper. In Section 3 we
prove a restriction theorem for strongly stable sheaves with
vanishing discriminant. The results of this section are used in
Sections 4, 5 and 10. In Section 4 we prove the analogue of
Simpson's theorem in positive characteristic. In Section 5 we
prove that reflexive strongly semistable sheaves with vanishing
Chern classes are numerically flat locally free sheaves. In
Section 6 we finally define the S-fundamental group scheme and we
compare it to other fundamental group schemes. In Section 7 er
study numerically flat principal bundles and we state some results
generalizing the results on the monodromy group proved in
\cite{BPS}. In Section 8 we study basic properties of the
S-fundamental group scheme. In Section 9 we prove some vanishing
theorems for the first and second cohomology groups of sheaves
associated to twists of numerically flat sheaves. Finally, in
Section 10 we prove Lefschetz type theorems for the S-fundamental
group scheme.

\medskip

After this paper was written, there appeared preprint \cite{BP} of V.
Balaji and A.J. Parameswaran. In this paper the authors introduce another
graded Tannaka category of  vector bundles with filtrations whose
quotients are degree $0$ stable, strongly semistable vector
bundles. The zeroth graded piece of their construction corresponds
to our S-fundamental group scheme. However, unlike our group
scheme their group scheme depends on the choice of polarization.

After the author send this paper to V. B. Mehta, he obtained in
return another preprint \cite{Me}. In this paper Mehta also introduces the
S-fundamental group scheme (using numerically flat bundles and
calling it the ``big fundamental group scheme''). He proves that
if $G$ is semisimple then principal $G$-bundles whose pull backs
to all curves are semistable come from a representation of the
S-fundamental group scheme (see [Theorem 5.8, loc. cit.]). He also
shows that for a smooth projective variety defined over an
algebraic closure of a finite field the S-fundamental group scheme
is isomorphic to  Nori's fundamental group scheme (see [Remark
5.11, loc. cit.]).

\subsection{Notation and conventions}

For simplicity all varieties in the paper are defined over an
algebraically closed field $k$. We could also assume that $k$ is
just a perfect field but in this case our fundamental group,
similarly to Nori's fundamental group, is not a direct
generalization of Grothendieck's fundamental group as it ignores
the arithmetic part of the group. Let us also recall that if a
variety is defined over a non-algebraically closed field $k$, then
the notions of stability and semistability can be also defined
using subsheaves defined over $k$. In case of semistability this
is equivalent to geometric semistability (i.e., we can pass to the
algebraic closure and obtain the same notion), but this is no
longer the case for stability (see \cite[Corollary 1.3.8 and
Example 1.3.9]{HL}).

We will not need to distinguish between absolute and geometric
Frobenius morphisms.

Let $E$ be a rank $r$ torsion free sheaf on a smooth
$n$-dimensional projective variety $X$ with an ample line bundle
$L$. Then one can define the \emph{slope} of $E$ by $\mu
(E)=c_1E\cdot c_1 L ^{n-1}/r$. The \emph{discriminant} of $E$ is
defined by $\Delta (E)=2rc_2(E)-(r-1)c_1^2(E)$.

One can also define a generalized slope for pure sheaves for
singular varieties but the notation becomes more cumbersome and
for simplicity of notation we restrict only to the smooth case.

Semistability  will always mean slope semistability with respect
to the considered ample line bundle (or a collection of ample line
bundles). The slope of a maximal destabilizing subsheaf of $E$ is
denoted by $\mu _{\max}(E)$ and that of minimal destabilizing
quotient by $\mu _{\min} (E)$.

In the following we identify locally free sheaves and
corresponding vector bundles.

Let us recall that an affine $k$-scheme $\Spec A$ is called
\emph{algebraic} if $A$ is finitely generated as a $k$-algebra.

In this paper all representations of groups are continuous. In
other words, all groups in the paper are pro-algebraic so we have
a structure of a group scheme and the homomorphism defining the
representation is required to be a homomorphism of group schemes.

\section{Preliminaries}

In this section we gather a few auxiliary results.

\subsection{Numerical equivalence}

Let $X$ be a smooth complete $d$-dimensional variety defined over
an algebraically closed field $k$. Then an $e$-cycle $\alpha$ on
$X$ is \emph{numerically equivalent to zero} if and only if $\int
_X\alpha \beta=0$ for all $(d-e)$-cycles $\beta$ on $X$. Let $\Num
_*X$ be the subgroup of the group of cycles $Z_*X$ generated by
cycles numerically equivalent to $0$. Then $N_*X=Z_*X/\Num_*X$ is
a finitely generated free abelian group (see \cite[Examples 19.1.4
and 19.1.5]{Fu}).

In this paper, Chern classes of sheaves will be considered only as
elements of $N_*X$.

Similarly as above  one defines the \emph{numerical Grothendieck
group} $K(X)_{\num}$ as the Grothendieck group (ring) $K(X)$ of
coherent sheaves modulo numerical equivalence, i.e., modulo the
radical of the quadratic form given by the Euler characteristic
$(a,b)\to \chi (a\cdot b)=\int _X \ch (a)\ch (b)\td (X)$. Here
$\ch : K(X)_{\num}\otimes \QQ\to N_*(X)\otimes {\QQ}$ is the map
given by the Chern character. By $\ch _i$ we denote the degree $i$
part of this map.
\medskip

The following result is well known but the author was not able to
provide a reference to its proof and hence we give it below:

\begin{Lemma} \label{obvious}
If a family of isomorphism classes of sheaves on $X$ is bounded
then the set of Chern classes of corresponding sheaves is finite.
\end{Lemma}

\begin{proof}
By definition a family is bounded if there exists a $k$-scheme $S$
of finite type and a coherent $\O_{S\times X}$-module $\F$ such
that $\{ \F_{s\times X}\} _s\in S$ contains all members of this
family. Passing to the flattening stratification of $S$ for $\F$
(see, e.g., \cite[Theorem 2.15]{HL}) we can assume that $\F$ is
$S$-flat. Let $q: S\times X\to X$ be the projection. For a flat
family $\F$ the Euler characteristic $s\to \chi ((\F\otimes
q^*\alpha)_s)$ is locally constant for all classes $\alpha \in
K(X)$. This implies that there are only finitely many classes of
$[\F_s]$ in $K(X)_{\num}$. Since $\ch : K(X)_{\num}\otimes \QQ\to
N_*(X)\otimes {\QQ}$ is an isomorphism and $N_*(X)$ is torsion
free we get the required assertion.
\end{proof}

\medskip

\subsection{Nefness}

Let us recall that a locally free sheaf $E$ on a complete
$k$-scheme is called \emph{nef} if and only if $\O_ {\PP (E)}(1)$
is nef on the projectivization $\PP (E)$ of $E$. We say that $E$
is \emph{numerically flat} if both $E$ and $E^*$ are nef.

A locally free sheaf $E$ is nef if and only if for any finite
morphism $f: C\to X$ from a smooth projective curve $C$ we have
$\mu _{\min} (f^*E)\ge 0$ (see, e.g., \cite[Theorem 2.1 and
p.~437]{Ba}). Hence, quotients of a nef bundle are nef.

Let $f: X\to Y$ be a surjective morphism of complete
$k$-varieties. Then $E$ on $Y$ is nef if and only if $f^*E$ is
nef.  Similarly, since pull back commutes with dualization, $E$ is
numerically flat if and only if $f^*E$ is numerically flat.

\medskip

\subsection{Flatness and complex fundamental groups}

Let us recall that a {\sl flat bundle} on a complex manifold is a
$\C^{\infty}$ complex vector bundle together with a flat
connection. One can also look at it as a complex representation of
the topological fundamental group $\pi_1(X,x)$ or a bundle
associated to a local system of complex vector spaces. We say that
a flat bundle is {\sl unitary} if it is associated to a
representation that factors through the unitary group. For unitary
flat bundles (and extensions of unitary flat bundles) the
holomorphic structure is preserved in the identification of flat
bundles and Higgs bundles.

The following theorem was proven in the curve case by
Narasimhan--Seshadri, and then generalized by Donaldson,
Uhlenbeck--Yau and Mehta--Ramanathan to higher dimension:

\begin{Theorem} \textup{(see \cite[Theorem 5.1]{MR})}
Let $X$ be a smooth $d$-dimensional complex projective manifold
with an ample divisor $H$. Let $E$ be a vector bundle on $X$ with
$c_1(E)=0$ and $c_2(E)H^{d-2}=0$. Then $E$ comes from an
irreducible unitary representation of $\pi _1(X,x)$ if and only if
$E$ is slope $H$-stable.
\end{Theorem}

Later C. Simpson generalized this statement to correspondence
between flat bundles and semistable Higgs bundles. As a special
case he obtained the following result:

\begin{Theorem} \label{Simp} \textup{(\cite[Corollary 3.10 and the following
remark]{Si})}  There exists an equivalence of categories between
the category of holomorphic flat bundles which are extensions of
unitary flat bundles and the category of $H$-semistable bundles
with $ch_1\cdot H^{d-1}=ch_2\cdot H^{d-2}=0$. In particular, the
latter category does not depend on the choice of ample divisor
$H$.
\end{Theorem}

Let us fix a point $x\in X$. Then the above category of
$H$-semistable bundles $E$ with $ch_1(E)H^{d-1}=ch_2(E)H^{d-2}=0$
can be given the structure of a neutral Tannakian category (cf.
\cite[p. 70]{Si}) with a fibre functor defined by sending bundle
$E$ to its fibre $E(x)$.

\begin{Definition}
The affine group scheme over $\CC$ corresponding to the above
Tannakian category is called the \emph{S-fundamental group scheme}
and denoted by $\pi_1^S(X,x)$.
\end{Definition}

\medskip
In \cite[Section 5]{Si} Simpson defined the \emph{universal
complex pro-algebraic group} $\pi_1^a(X,x)$ as the inverse limit
of the directed system of representations $\rho : \pi_1(X,x)\to G$
for complex algebraic groups $G$, such that the image of $\rho$ is
Zariski dense in $G$ (in the language of \cite[10.24]{De}
$\pi_1^a(X,x)$ is an \emph{algebraic envelope} of  the topological
fundamental group). This group is Tannaka dual to the neutral
Tannaka category of semistable Higgs bundles with vanishing
(rational) Chern classes (and with the obvious fibre functor of
evaluation at $x$). Therefore by \cite[Proposition 2.21 (a)]{DM}
we get the following corollary which solves the problem posed in
\cite[Remark 5.2]{BPS}:

\begin{Corollary} \label{complex-surjection}
We have a surjection $\pi_1^a(X,x)\to \pi_1^S(X,x)$ of
pro-algebraic groups (or, more precisely, a faithfully flat
morphism of complex group schemes).
\end{Corollary}

In general, the surjection $\pi_1^a(X,x)\to \pi_1^S(X,x)$ is not
an isomorphism. For example, it is not an isomorphism for all
curves of genus $g\ge 2$ because $\O_C\oplus\omega_C$ with the
Higgs field given by the identity on $\omega_C$ is Higgs
semistable but not semistable (after twisting by an appropriate
line bundle this gives a representation of $\pi_1^a(X,x)$ not
coming from $\pi_1^S(X,x)$).

If $\mu_{\max} (\Omega _X)<0$ then $\pi_1^a(X,x)\to \pi_1^S(X,x)$
is an isomorphism. This follows from the fact that if $\mu_{\max}
(\Omega _X)<0$ then all (Higgs) semistable Higgs bundles have
vanishing Higgs field and they are semistable in the usual sense.
In fact, $\pi_1^a(X,x)$ and $\pi_1^S(X,x)$ are both zero by the
following lemma:

\begin{Lemma}
If $X$ is a complex manifold with $\mu _{\max} (\Omega_X)<0$ then
$\pi_1^a(X,x)=0$.
\end{Lemma}

\begin{proof}
By assumption $h^i (X, \O_X)=h^0(X, \Omega_X^i)=0$ for $i>0$.
Therefore $\chi (X, \O_X)=1$. Let $\pi : Y\to X$ be an \'etale
cover. Then $\mu_{\max} (\Omega _Y)<0$ (because $\Omega _Y=\pi ^*
\Omega _X $) so $\chi (Y, \O_Y)=1$. But $\chi (Y, \O_Y)=\deg \pi
\cdot \chi (X, \O_X)$ so $\pi$ is an isomorphism. This implies
that the \'etale fundamental group of $X$ is trivial. But by
Malcev's theorem a finitely generated linear group is residually
finite so any non-trivial representation  $\pi _1(X,x)\to G$ in an
algebraic complex affine group gives rise to some non-trivial
representation in a finite group. Therefore $\pi_1^a(X,x)$ is also
trivial.
\end{proof}

\medskip

Note that assumption immediately implies that $H^0(X, \Omega
_X^{\otimes m})=0$ for $m>0$. There is a well-known Mumford's
conjecture (see, e.g., \cite[Chapter IV, Conjecture 3.8.1]{Ko})
saying that in this case $X$ should be rationally connected. Since
rationally connected complex manifolds are simply connected we
expect that all varieties in the lemma are simply connected.

\section{Deep Frobenius descent in higher dimensions}

The aim of this section is to recall some boundedness results used
later in several proofs, and to generalize some results of H.
Brenner and A. Kaid \cite{BK}  to higher dimensions.

\medskip

Let $f:\X\to S$ be a smooth projective morphism of relative dimension
$d\ge 1$ of schemes of finite type over a fixed noetherian ring $R$.
Let $\O _{\X/S}(1)$ be an $f$-very ample line bundle on $\X$.
Let $\cT (r,c_1,\Delta;\mu_{\max})$ be the family of isomorphism classes of sheaves $E$ such that
\begin{enumerate}
\item $E$ is a rank $r$ reflexive sheaf on a fibre $\X_s$ over some point $s\in S$.
\item Let $H_s$ be some divisor corresponding to the restriction of $\O_{\X /S}(1)$ to $\X_s$.
Then $c_1(E)H_s^{d-1}=c_1$ and $(\Delta (E)-(c_1(E)-r/2K_X ) ^2)
H_s^{d-2}\le \Delta$.
\item $\mu_{\max}(E)\le \mu_{max}$.
\end{enumerate}

\medskip

The following theorem is a special case of \cite[Theorem
3.4]{La2}. The only difference is that we allow mixed
characteristic. The proof of the theorem holds in this more
general case because the dependence on the characteristic is very
simple (see the proof of \cite[Theorem 4.4]{La1}).

\begin{Theorem} \label{refl-bound}
The family $\cT _{\X/S }(r,c_1,\Delta;\mu_{\max})$ is bounded. In
particular, the set of Hilbert polynomials of sheaves in
$\cT_{\X/S}(r,c_1,\Delta;\mu_{\max})$ is finite. Moreover, there
exist polynomials $P_{\X, S}$, $Q_{\X /S}$ and $R_{\X/S}$ such
that for any $E\in \cT_{\X/S}(r,c_1,\Delta;\mu_{\max})$ we have:
\item{(1)}
$E(m)$ is $m$-regular for $m\ge P_{\X/S}(r, c_1, \Delta,\mu_{\max})$,
\item{(2)}
$H^{1}(X, E(-m))=0$ for  $m\ge Q_{\X/S}(r, c_1,\Delta ,\mu_{\max})$,
\item{(3)}
$h^1 (X, E(m))\le R_{X/S}(r, c_1, \Delta ,\mu_{\max})$ for all $m$.
\end{Theorem}

\begin{Example}
Let $C$ be a smooth projective curve of genus $g\ge 1$. Let
$p_1,p_2$ denote projections  of $C\times C$ on the corresponding
factors. Let us fix a point $x\in C$  and put $H=p_1^*x+p_2^*x$.
Let $\Delta\subset C\times C$ be the diagonal. Finally, set
$L_n=\O_{C\times C}(n(H-\Delta))$. Then $c_1(L_n)H=0$ and $\Delta
(L_n)=0$ but the family $\{L_n\} _{n\in \ZZ}$ is not bounded. This
shows that in the definition of the family $\cT
(r,c_1,\Delta;\mu_{\max})$ we cannot replace the bound on $(\Delta
(E)-(c_1(E)-r/2K_X ) ^2) H^{d-2}$ with the bound on $\Delta
(E)H^{n-2}$ as the family need not be bounded.
\end{Example}

The following corollary generalizes \cite[Lemma 3.2]{BK}:

\begin{Corollary} \label{const}
There exists some constant $c=c(\X /S, r,c_1, \Delta ;\mu_{\max})$
such that for any (possibly non-closed) point $s\in S$ the number
of reflexive sheaves $E$ of rank $r$ with fixed
$c_1(E)H^{d-1}=c_1$, $(\Delta (E) -(c_1(E)-r/2K_X ) ^2)H^{d-2}\le
\Delta$ and $\mu_{\max}(E)\le \mu_{max}$ is bounded from above by
$|k(s)|^c$.
\end{Corollary}

\begin{proof}
By the above theorem there are only finitely many possibilities
for the Hilbert polynomial of $E$, so we can fix it throughout the
proof. Let us take $E$ as above on the fibre $X_s$ over a point
$s\in S$ with finite $k(s)$ (if $k(s)$ is infinite then our assertion
is trivially satisfied). By the above theorem if we take
$m=P_{\X/S}(r, c_1, \Delta,\mu_{\max})+1$ then $E(m)$ is globally
generated by $a=P(E)(m)$ sections. Let us define $E'$ using the
sequence
$$0\to E'\to \O _{X_s}(-mH_s)^{a}\to E\to 0.$$
Clearly, the Hilbert polynomial of $E'$ depends only on the
Hilbert polynomials of  $E$ and $H_s$. Since $\mu _{\max}(E')\le
\mu (\O _{X_s}(-mH_s))=-mH_s^d$ we can again use the above theorem
to find some explicit $m'$ such that $E'(m')$ is globally
generated by $b=P(E')(m')= a\chi (\O _{X_s}((m'-m)H))- P(E)(m')$
sections. Therefore $E$ is a cokernel of some map
$$\O _{X_s}(-m'H_s)^{b}\to \O _{X_s}(-mH_s)^{a}.$$
Then we can conclude similarly as in the proof of \cite[Lemma
3.2]{BK}. Namely, we can assume that the dimension of
$H^0(\O_{X_s}((m'-m)H_s))$ is computed by the Hilbert polynomial
of $\O_{X_s}$ (possibly we need to increase $m'$ but only by some
function depending on $\X /S$: for example we can apply the above
theorem to the rank $1$ case). Then the number of the sheaves that
we consider is bounded from the above by $|k(s)|^c$, where
$c=ab\chi(\O_X((m'-m)H))$.
\end{proof}

\medskip

Let $R$ be a $\ZZ$-domain of finite type containing $\ZZ$.
Let $f:\X\to S=\Spec R$ be a smooth projective morphism of relative dimension $d\ge 1$
and let $\O _{\X} (1)$ be an $f$-very ample line bundle.

Let $K$ be the quotient field of $R$. Let $\X_0=\X\times _S\Spec
K$ be the generic fibre of $f$. Let $\E$ be an $S$-flat family of
rank $r$ torsion free sheaves on the fibres of $f$. Let us choose
an embedding $K\subset \CC$.  Then for the restriction  $\E_0$ of
$\E$ to $\X_0$ we consider $\E_{\CC}=\E_0\otimes \CC$.

We say that $(s _n,e_n )_{n\in \NN}$, where $s _n \in S$ are
closed points and $e_n$ are positive integers, is a {\sl Frobenius
descent sequence} for $\E$ if there exist coherent sheaves $\F _n$
on the fibres $\X _{s _n}$ such that $\E _{\X _{s_n}}\simeq
(F^{e_n})^*\F _n$.

\medskip

The following theorem generalizes \cite[Theorem 3.4]{BK} to higher
dimensions and relates the notion of flatness in positive
characteristic to the one coming from complex geometry:

\begin{Theorem} Let us assume that there exists a Frobenius descent
sequence $(s_n,e_n )_{n\in \NN}$ for $\E$ with $(e_n-| k(s
_n)|^c)_{n\in \NN} \to \infty$, where $c$ is the constant from
Corollary \ref{const}. Then the restriction  $\E_0$ of $\E$ to the
generic fibre of $f$ is an extension of stable (with respect to an
arbitrary polarization) locally free sheaves with vanishing Chern
classes. Moreover, $\E_{\CC}$ is also an extension of slope stable
locally free sheaves with vanishing Chern classes (note that these
stable sheaves need not be extensions of sheaves defined over
$K$). In particular, $\E_{\CC}$ has structure of a holomorphic
flat bundle on $X_{\CC}$ which is an extension of unitary flat
bundles.
\end{Theorem}

\begin{proof} Note that we can assume that $S$ is connected. Then by
$S$-flatness of $\E$ the numbers $c_i=c_i(\E _s)\cdot c_1(\O_{\X
_s}(1))^{d-i}$ are independent of $s\in S$. Since
$$c_i(\E _{s_n})\cdot c_1(\O_{\X _{s_n}}(1))^{d-i}=(\chr k(s_n))^{e_n}
c_i(\F _{n})\cdot c_1(\O_{\X _{s_n}}(1))^{d-i}$$ and $e_n\to
\infty$ we see that $c_i=0$. The rest of the proof is the same as
the proof of \cite[Theorem 3.4]{BK} using Corollary \ref{const}
instead of \cite[Lemma 3.2]{BK}. The final part of the theorem
follows from \cite[Theorem 2]{Si} and \cite[Lemma 3.5]{Si}.

Alternatively, we can use Theorem \ref{loc-free} as for large $n$
the sheaves $\E_{s_n}$ are strongly semistable as follows from the
proof. Hence by Theorem \ref{loc-free} $\E_{s_n}$ are locally free
for large $n$ which implies that $\E _0$ is locally free by
openness of local freeness. Then one can consider the
Jordan--H\"older filtration of $\E_0$, extend it to some
filtration over nearby fibres and use induction on the rank as in
the proof of Theorem \ref{loc-free}.
\end{proof}

\bigskip

\section{Restriction theorem for strongly stable sheaves
with vanishing discriminant}

In this section we prove the restriction theorem for strongly
stable sheaves. It is used in the next section and it also plays
an important role in proofs of the Lefschetz type theorems for the
S-fundamental group (see, e.g., proof of Theorem
\ref{Lefschetz1}).

\medskip

Let us consider $\PP ^2$ over an algebraically closed field of
characteristic $p>0$. In \cite{Br} H. Brenner showed that the
restriction of $\Omega _{\PP ^2}$ to a curve $x^d+y^d+z^d=0$,
where $p^{e}<d<3/2p^{e}$ for some integer $e$, is not strongly
stable. Hence the restriction of a strongly stable sheaf to a
smooth hypersurface of large degree need not be strongly stable.
But by \cite[Theorem 5.2]{La1} restriction of a strongly stable
sheaf with trivial discriminant to a hypersurface of large degree
is still strongly stable (the bound on the degree of this
hypersurface depends on the rank of the sheaf). However, in this
case we have the following stronger version of restriction theorem
(valid in arbitrary characteristic):

\begin{Theorem} \label{bogomolov}
Let $D_1,\dots ,D_{d-1}$ be a collection of ample divisors on $X$
of dimension $d\ge 2$. Let $E$ be a rank $r\ge 2$ torsion free
sheaf with $\Delta (E)D_2\dots D_{d-2}=0$. Assume that $E$ is
strongly $(D_1,\dots ,D_{d-1})$-stable. Let $D \in |D_1|$ be any
normal effective divisor such that $E_D$ has no torsion. Then
$E_D$ is strongly $(D_2,\dots ,D_{d-1})_D$-stable.
\end{Theorem}

\begin{proof}
For simplicity of notation we proof  the result in case when all
the divisors $D_1,\dots ,D_{d-1}$ are equal to one ample divisor
denoted by $H$. The general proof is exactly the same.

Let $\Delta (E)H^{d-2}=0$ and assume that $E$ is strongly
$H$-stable. Let $D \in |H|$ be any normal effective divisor such
that $E_D$ has no torsion. We need to prove that $E_D$ is strongly
$H_D$-stable. Suppose that there exists a non-negative integer
$k_0$ such that the restriction  of ${\ti E}=(F^{k_0})^* E$ to $D$
is not stable. Let $S$ be a rank $\rho$ saturated destabilizing
subsheaf of ${\ti E}_D$. Set $T=(\ti E_D)/S$. Let $G$ be the
kernel of the composition $\ti E\to \ti E_D\to T$. From the
definition of $G$ we get a short exact sequence:
$$0\to G\to \ti E \to T\to 0.$$
Applying the snake lemma to the diagram
$$ \xymatrix{
&&0\ar[r]\ar[d]&0\ar[r]\ar[d]&S\ar[d]&\\
&0\ar[r]& \ti E (-D)\ar[r]\ar[d]& \ti E \ar[d]\ar[r]& \ti E_D\ar[d]\ar[r]& 0\\
&0\ar[r]& G\ar[r]&\ti E \ar[r]& T\ar[r]& 0\\
}$$ we also get the following exact sequence:
$$0\to \ti E(-D)\to G \to S\to 0.$$

 Computing $\Delta (G)$ we get
$$\Delta (G)H^{d-2}=-\rho (r-\rho) H^{d} +2 (r c_1(T)-(r-\rho)Dc_1(\ti E)) H^{d-2}. $$
By assumption $(r c_1(T)-(r-\rho)Dc_1(\ti E))H^{d-2}\le 0$, so
$$\Delta (G) H^{d-2} \le -\rho (r-\rho) H^{d}.$$
By  \cite[Theorem 2.7]{La1}, for large $l$ we have $\mu _{\max}
((F^l)^*G)=L_{\max} ((F^l)^*G)$ and similarly for $\mu _{\min}$.
Using strong $H$-stability of $\ti E$ and $\ti E(-D)$ we get for
large integers $l$
$$L _{\max} ((F^l)^*G)-\mu ((F^l)^*G)=\mu _{\max} ((F^l)^*G)
-\mu ((F^l)^*\ti E)+\frac{r-\rho}{r}p^l H^{d}\le {r-\rho \over
r}p^lH^d -{1\over r(r-1)}$$ and
$$\mu ((F^l)^*G)-L _{\min} ((F^l)^*G)=\mu ((F^l)^*\ti E(-D))-
\mu _{\min} ((F^l)^*G)+\frac{\rho}{r}p^lH^{d}\le
\frac{\rho}{r}p^lH^d -\frac{1}{r(r-1)}.$$ Hence, applying
\cite[Theorem 5.1]{La1} to $(F^l)^*G$ gives
\begin{eqnarray*}
0&\le H^d\cdot \Delta ((F^l)^*G)H^{d-2} +r^2(L_{\max} ((F^l)^*G)
-\mu
(F^l)^*G))(\mu ((F^l)^*G)-L_{\min}((F^l)^*G)) \\
&\le -\rho (r-\rho) p^{2l} (H^d)^2 +r^2 \left({r-\rho \over
r}p^lH^d -\frac {1}{r(r-1)}\right) \left(\frac{\rho}{r}p^lH^d
-{1\over r(r-1)}\right).
\end{eqnarray*} Therefore
$$\frac{r}{r-1}p^lH^d\le \frac{1}{(r-1)^2},$$
which gives a contradiction.
\end{proof}

Later we show a much stronger restriction theorem (see Corollary
\ref{strong-bogomolov}) but we need this weaker result to
establish Theorem \ref{loc-free} used in the proof of this
stronger result.

\section{Strongly semistable sheaves with vanishing Chern classes}

In this section we show that strongly semistable torsion free
sheaves with vanishing Chern classes are locally free and that
they are strongly semistable with respect to all polarizations.

\medskip

The following theorem is an analogue of \cite[Theorem 2]{Si} in
positive characteristic. However, we need a different proof as
Simpson's proof uses Lefschetz hyperplane theorem for topological
fundamental groups and the correspondence between flat (complex)
bundles and semistable Higgs bundles with vanishing Chern classes
(see \cite[Lemma 3.5]{Si}). We reverse his ideas and we use this
result to prove Lefschetz type theorems for \'etale, Nori and
S-fundamental groups.

\begin{Theorem} \label{loc-free}
Let $X$ be a smooth $d$-dimensional projective variety over an
algebraically closed field $k$ of characteristic $p>0$ and let $H$
be an ample divisor on $X$. Let $E$ be a strongly $H$-semistable
torsion free sheaf  on $X$ with $\ch _1(E)\cdot H^{d-1}=0$ and
$\ch _2 (E)\cdot H^{d-2}=0$. Assume that either $E$ is reflexive
or the reduced Hilbert polynomial of $E$ is equal to the Hilbert
polynomial of $\O_X$. Then $E$ is an extension of stable and
strongly semistable locally free sheaves with vanishing  Chern
classes. Moreover, there exists $n$ such that $(F^n)^*E$ is an
extension of strongly stable locally free sheaves with vanishing
Chern classes.
\end{Theorem}

\begin{proof}  Before starting the proof
of the theorem let us prove the following lemma:

\begin{Lemma}
Let $E$ be a strongly $H$-semistable torsion free sheaf  on $X$
with $\ch _1(E)\cdot H^{d-1}=0$ and $\ch _2 (E)\cdot H^{d-2}=0$.
Then the $1$-cycle $c_1(E)H^{d-2}$ is numerically trivial and
$\Delta (E)H^{d-2}=0$.
\end{Lemma}

\begin{proof}
By \cite[Theorem 3.2]{La1} we have $\Delta (E)H^{d-2} \ge 0$.
Therefore by the Hodge index theorem
$$0=2r( {\ch } _2 (E)H^{d-2})=(c_1(E)^2- \Delta (E))H^{d-2}\le c_1(E)^2 H^{d-2} \le \frac{(c_1(E)H^{d-1})^2}{H^d}=0,$$
which implies the required  assertions.
\end{proof}

In case of curves the theorem follows from the existence of the
Jordan--H\"older filtration. The proof is by induction on the
dimension starting with dimension $2$.

If $X$ is a surface then we prove that a strongly semistable
torsion free sheaf $E$ on $X$ with $\ch _1(E)\cdot H=0$ and $\ch
_2 (E)=0$ is an extension of stable and strongly semistable
locally free sheaves with vanishing Chern classes. This part of
the proof is well known and analogous to the proof of
\cite[Theorem 2]{Si}. Namely, the reflexivization $E^{**}$ is
locally free and strongly semistable. Hence by \cite[Theorem
3.2]{La1} $\Delta (E^{**})\ge 0$. Since $\Delta (E^{**})\le \Delta
(E)$ and by the above lemma $\Delta (E)=0$, we have
$c_2(E^{**}/E)=0$. This implies that $E^{**}/E$ is trivial and $E$
is locally free. The required assertion follows easily from this
fact (it will also follow from the proof below).

Now fix $d\ge 3$ and assume that the theorem holds in dimensions
less than $d$. Let $E$ be a strongly stable reflexive sheaf on
$d$-dimensional $X$ with $\ch _1(E)\cdot H^{d-1}=0$ and $\ch _2
(E)\cdot H^{d-2}=0$. Then by the above lemma all the sheaves
$\{(F^n)^*E\} _{n\in \NN}$ are in the family $\cT _{X/k }(r,0,0;
0)$. This family is bounded by Theorem \ref{refl-bound}.
Therefore, since by Lemma \ref{obvious} there are only finitely
many classes among $c_i((F^n)^*E)=p^n c_i(E)$, we see that the
Chern classes of $E$ vanish. In particular, for any smooth divisor
$D$ on $X$ the reduced Hilbert polynomial of $E_D$ is equal to the
Hilbert polynomial of $\O_D$. Let us also remark that $E_D$ is
torsion free (see, e.g., \cite[Corollary 1.1.14]{HL}).

Let us first assume that $E$ is strongly stable. By Theorem
\ref{bogomolov} the restriction $E_D$ is also strongly stable for
all smooth divisors $D\in |mH|$ and all $m\ge 1$. In particular,
$E_D$ is locally free by the induction assumption. Note that if
$x\in D$ then $E\otimes k(x)\simeq E_D\otimes k(x)$ is an
$r$-dimensional vector space over $k(x)\simeq k$. Therefore by
Nakayama's lemma $E$ is locally free at $x$. By Bertini's theorem
(see, e.g., \cite[Theorem 3.1]{DH}) for any closed point $x\in X$
there exists for large $m$ a smooth hypersurface $D\in |mH|$
containing $x$. Therefore $E$ is locally free at every point of
$X$, i.e., it is locally free.

Now let us consider the general case. Let us choose $m$ such that
all quotients in a Jordan-H\"older filtration of $(F^m)^*E$ are
strongly stable (clearly such $m$ exists). Then we can prove the
result by induction on the rank $r$. Namely, if
$$0=E_0\subset E_1\subset \dots \subset E_l=(F^m)^*E$$
is the Jordan-H\"older filtration then $E_1$ is reflexive with $c
_1(E_1)H^{d-1}=0$ and $\Delta (E_1) H^{d-2}=0$. The last equality
follows from Bogomolov's inequality for strongly semistable
sheaves (see [La1, Theorem 3.2]) and from the inequality $\Delta
(E_1) H^{d-2}\le \Delta (E) H^{d-2}$ obtained from the Hodge index
theorem (see, e.g., \cite[Corollary 7.3.2]{HL}). So by the above
we know that $E_1$ is locally free with vanishing Chern classes.
Note that $\{(F^n)^*(((F^m)^*E)/E_1)\} _{n\in \NN}$ are semistable
torsion free quotients of the sheaves from a bounded family.
Therefore by Grothendieck's lemma (see \cite[Lemma 1.7.9]{HL})
they also form a bounded family and by the previous argument they
have vanishing Chern classes. Hence the reduced Hilbert polynomial
of $((F^m)^*E)/E_1$ is equal to the Hilbert polynomial of $\O_X$
and we can apply the induction assumption to conclude that
$((F^m)^*E)/E_1$ is locally free. This implies that all the
quotients in the Jordan-H\"older filtration of $(F^m)^*E$ are
locally free, which proves the last assertion of the theorem. Then
the first assertion follows just by taking any Jordan-H\"older
filtration of $E$.

Now we assume that the reduced Hilbert polynomial of $E$ is equal
to the Hilbert polynomial of $\O_X$ but we do not assume that $E$
is reflexive. Then the reflexivization $E^{**}$ of $E$ satisfies
the previous assumptions and hence it is locally free with
vanishing Chern classes. Therefore the reduced Hilbert polynomial
of $E^{**}$ is also equal to the Hilbert polynomial of $\O_X$. In
particular, the Hilbert polynomial of the quotient $T=E^{**}/E$ is
trivial and hence $T=0$ and $E$ is reflexive. So we reduced the
assertion to the previous case (without changing the rank which is
important because of the induction step).
\end{proof}

\medskip
Note that the theorem fails if $d\ge 3$ and we do not make any
additional assumptions on the Hilbert polynomial or reflexivity of
$E$. For example one can take the ideal sheaf of a codimension
$\ge 3$ subscheme. This sheaf is strongly stable and torsion free
with $\ch _1(E)\cdot H^{d-1}=0$ and $\ch _2 (E)\cdot H^{d-2}=0$
but it is not locally free.

\medskip

\begin{Corollary} \label{bogomolov-cor}
Let $E$ be a locally free sheaf with $\ch _1(E)\cdot H^{d-1}=0$
and $\ch _2 (E)\cdot H^{d-2}=0$. Let $D \in |H|$ be any normal
effective divisor. If $E$ is strongly semistable then $E_D$ is
strongly semistable.
\end{Corollary}

\begin{proof}
By the above theorem we can choose $m$ such that all quotients in
a Jordan-H\"older filtration of $(F^m)^*E$ are locally free and
strongly stable. Then by Theorem \ref{bogomolov} the restriction
of each quotient is strongly stable which proves the corollary.
\end{proof}

\begin{Remark}
Let us remark that in general a strongly semistable locally free
sheaf on a smooth projective variety does not restrict to a
semistable sheaf on a general smooth hypersurface of large degree
passing through a fixed point (not even in characteristic $0$).

For example one can take a non-trivial extension $E$ of $m_x$ by
$\O_{\PP ^2}$ for some $x\in \PP^2$. Then $E$ is a strongly
semistable locally free sheaf  but the restriction of $E$ to any
curve passing through $x$ is not semistable.

This shows that one cannot generalize the proof of
Mehta--Ramanathan's theorem to prove stability of the restriction
of a stable sheaf to a general hyperplane passing through some
fixed points (the proof for restriction of stable sheaves uses
restriction of semistable sheaves).
\end{Remark}

\medskip
The following theorem says that strong semistability for locally
free sheaves with vanishing Chern classes does not depend on the
choice of polarization:

\begin{Proposition} \label{independence}
Let $D_1,\dots, D_{d-1}$ be ample divisors on $X$. Let $E$ be a
strongly $(D_1,\dots, D_{d-1})$-semistable reflexive sheaf on $X$
with $\ch _1(E)\cdot D_1\dots D_{d-1}=0$ and $\ch _2 (E)\cdot
D_2\dots D_{d-1}=0$. Then it is locally free with vanishing Chern
classes and it is strongly semistable with respect to an arbitrary
collection of ample divisors.
\end{Proposition}

\begin{proof}
The first assertion can be proven as in the above theorem.  So it
is sufficient to prove that for any ample divisor $A$ the sheaf
$E$ is strongly $(A,D_2,\dots, D_{d-1})$-semistable. We can assume
that $D_2,\dots D_{d-1}$ are very ample. Taking a general complete
intersection of divisors in $|D_2|,\dots, |D_{d-1}|$ and using
Theorem \ref{bogomolov} we see that it is sufficient to prove the
assertion in  the surface case. In the following we assume that
$d=2$ and set $H=D_1$. Taking the Jordan--H\"older filtration of
some Frobenius pull-back of $E$ we can also assume that $E$ is in
fact strongly $H$-stable.

Let us consider the family $\F $ of all sheaves $E'$ such that
$\mu _A(E')>0$ and there exists a non-negative integer $n$ such
that $E'\subset (F^n)^*E$ and the quotient $(F^n)^*E/E'$ is
torsion free. Let us set $H_t=(1-t)H+tA$ for $t\in [0,1]$. Since
the family $\{(F^n)^*E\} _n$ is bounded, the family  $\F$ is also
bounded by Grothendieck's lemma \cite[Lemma 1.7.9]{HL}. Therefore
there exists the largest rational number $0<t<1$ such that for all
sheaves $E'\in \F$ we have $\mu_{H_t} (E')\le 0$ (note that by
assumption $\mu_H(E')<0$). Then there exists a sheaf $E'\in \F$
such that $\mu _{H_t}(E')=0$.

If $E'$ is not strongly $H_t$-semistable then there exist an
integer $l$ and a saturated subsheaf $E_1'\subset (F^l)^*E'$ such
that $\mu _{H_t} (E'_1)> \mu _{H_t} ((F^l)^* E')=0$. But $E'_1\in
\F$ so we have a contradiction with our choice of $t$. Therefore
$E'$ is strongly $H_t$-semistable.

Let us take integer $n_0$ such that $E'\subset (F^{n_0})^*E$.
Similarly as above we can prove that the quotient
$E''=(F^{n_0})^*E/E'$ is strongly $H_t$-semistable. Namely, if
$E''$ is not strongly $H_t$-semistable then there exist an integer
$l$ and a quotient sheaf $(F^l)^*E''\to E''_1$ such that $\mu
_{H_t} (E''_1)< \mu _{H_t} ((F^l)^* E'')=0$. But then the kernel
of $(F^{l+n_0})^* E\to E''_1$ is in $\F$ and it has positive slope
with respect to $H_t$ which contradicts our choice of $t$.

Therefore all the sheaves in the following exact sequence
$$0\to E'\to (F^{n_0})^*E \to E''\to 0$$
are strongly $H_t$-semistable with $H_t$-slope equal to $0$. Now
let us recall that by the Hodge index theorem we have
\begin{eqnarray*}
0={\Delta ((F^{n_0})^* E)\over r}&=& \frac{\Delta (E')}{r'}
+\frac{\Delta (E'')}{r''} -\frac{r'r''}{r}\left( \frac{c_1E'}
{r'}-\frac{c_1E''}{r''}\right) ^2\\
&\ge& \frac{\Delta (E')}{r'} +\frac{\Delta (E'')}{r''}
-\frac{r'r''}{rH_t^2}(\mu _{H_t}(E')
-\mu _{H_t}(E'')) .\\
\end{eqnarray*}
But by \cite[Theorem 3.2]{La1} we have $\Delta (E')\ge 0$, $\Delta
(E'')\ge 0$. Since $\mu _{H_t}(E')=\mu _{H_t}(E'')=0$ we see that
both $\Delta (E')$ and $\Delta (E'')$ are equal to $0$. Therefore
by Theorem \ref{loc-free}  both $E'$ and $E''$ have vanishing
Chern classes which contradicts strong $H$-stability of $E$.
\end{proof}

\medskip
\begin{Remark}
Note that nefness of $D_1$ is not sufficient to get the assertion
of the theorem. For example, if $X$ is a surface and $D_1$ is a
numerically non-trivial nef divisor with $D_1^2=0$ (e.g., a fibre
of a morphism of $X$ onto a curve) then the family
$\{\O_X(nD_1)\oplus \O_X(-nD_1)\} _{n}$ is not bounded although it
consists of strongly $D_1$-semistable locally free sheaves with
$\ch _1 \cdot D_1=0$ and $\ch _2 =0$.
\end{Remark}
\medskip

\section{Comparison with numerically flat bundles}

In this section we compare strongly semistable vector bundles with vanishing Chern classes
with  numerically flat vector bundles and we show that they
can be used to define a Tannaka category.

\medskip

Let $\Vect (X)$ denotes the full subcategory of the category of
coherent sheaves on $X$, having as objects all strongly
$H$-semistable reflexive sheaves with  $\ch _1(E)\cdot H^{d-1}=0$
and $\ch _2 (E)\cdot H^{d-2}=0$. By Proposition
\ref{independence}, $\Vect (X)$ does not depend on the choice of
$H$ so we do not include it into notation.

Let us mention that in the complex case the above category can be
identified with the category of numerically flat vector bundles
(see Theorem \ref{Simp} and \cite[Theorem 1.18]{DPS}). The author
does not know a direct purely algebraic proof of this equivalence
(over $\CC$). A similar characterization holds also in positive
characteristic:

\begin{Proposition} \label{num-nef}
Let $X$ be a smooth projective $k$-variety. Let $E$ be a coherent
sheaf on $X$. Then the following conditions are equivalent:
\begin{enumerate}
\item $E\in \Vect (X)$,
\item $E$ is numerically flat  (in particular, it is locally free),
\item $E$ is nef of degree $0$ with respect to some ample divisor
(in particular, it is locally free).
\end{enumerate}
\end{Proposition}

\begin{proof}
First we prove that 1 implies 2. If $E\in \Vect (X)$ then the
family $\{(F^n)^*E\} _n$ is bounded, so there exists an ample line
bundle $L$ on $X$ such that $(F^n)^*E\otimes L$ is globally
generated for $n=0,1,\dots$ Therefore for any smooth projective
curve $C$ and a finite morphism $f:C\to X$ the bundles
$f^*((F^n)^*E\otimes L)$ are globally generated. In particular,
$\mu_{\min}(f^*((F^n)^*E\otimes L))\ge 0$. Therefore for all $n\ge
0$
$$ -\deg f^*L\le \mu_{\min}(f^*((F^n)^*E))\le p^n \mu_{\min} (f^*E).$$
Dividing by $p^n$ and passing with $n$ to infinity we get $\mu
_{\min } (f^*E)\ge 0$. Therefore $E$ is nef. Since $E^*\in \Vect
(X)$, $E^*$ is also nef.

To prove that 2 implies 3 we take $E$ such that both $E$ and $E^*$
are nef.  Let us fix some ample divisor $H$ on $X$. Let us remark
that if some polynomial in Chern classes of ample vector bundle is
positive (see \cite[p.~35]{FL} for the definition) then it is also
non-negative for nef vector bundles. Therefore by \cite[Theorem
I]{FL} $c_1 \cdot H^{d-1}, c_2\cdot H^{d-2}, (c_1^2-c_2)\cdot
H^{d-2}$ are non-negative for all nef vector bundles. In
particular, from $c_1(E)H^{d-1}\ge 0$ and $c_1(E^*)H^{d-1}\ge 0$
we get $c_1(E)H^{d-1}= 0$.

To prove that 3 implies 1 note that $E$ is strongly semistable
with respect to all polarizations. By assumption and the Hodge
index theorem we have
$$0\le c_1^2(E)H^{d-2}\cdot H^d\le (c_1(E)H^{d-1})^2=0.$$
Hence from non-negativity of $c_2\cdot H^{d-2}, (c_1^2-c_2)\cdot
H^{d-2}$ we see that $c_2(E) H^{d-2}$ is equal to $0$. Therefore
by definition $E\in \Vect (X)$.
\end{proof}

\medskip

\begin{Remark}
Note that the condition that a locally free sheaf $E$ is
numerically flat is equivalent to the condition that for all
smooth curves $C$ and all maps $f:C\to X$ the pull back $f^*E$ is
semistable of degree $0$. In this case one sometimes says that $E$
is \emph{Nori semistable} (see, e.g., \cite{Me}). Obviously, this
is completely fair as Nori made huge contributions in the subject
but it should be noted that in \cite{No} Nori considered a
slightly different condition. Namely, he considered locally free
sheaves $E$ such that for all smooth curves $C$ and all birational
maps $f:C\to X$ onto its image, the pull back $f^*E$ is semistable
of degree $0$ (see \cite[p.~81, Definition]{No}). In positive
characteristic such sheaves do not form a tensor category.
\end{Remark}

\medskip

Note that the proof of the above proposition gives another proof
of Proposition \ref{independence}. As in the proof of Proposition
\ref{independence} we can restrict to the surface case so that we
deal with only one ample divisor when the above proof shows the
assertion (in general however, there are technical problems with
boundedness with respect to collection of polarizations).

\medskip

Proposition \ref{num-nef} allows us to define $\Vect (X)$ for
complete $k$-schemes. Then $\Vect (X)$ denotes the full
subcategory of the category of coherent sheaves on $X$, which as
objects contains all numerically flat locally free sheaves. If $X$
is smooth and projective then by Proposition \ref{num-nef} this
gives the same category as before.

The following corollary is a generalization of Theorem
\ref{bogomolov}:

\begin{Corollary} \label{strong-bogomolov}
\textup{(Very strong restriction theorem)} Let $X$ be a complete
$k$-scheme and let $E\in \Vect (X)$. Then for any closed subscheme
$Y\subset X$ the restriction $E_Y$ is in $ \Vect (Y)$.
\end{Corollary}

Clearly, $E\in \Vect (X)$ if and only if the restriction of $E$ to
every curve $C$  belongs to the category  $\Vect (C)$. This gives
relation with the category considered by Nori in \cite{No}.

By \cite[Proposition 3.5]{Ba} tensor product of two nef sheaves is
nef. Therefore we have the following corollary:

\begin{Corollary}
Let $X$ be a complete $k$-scheme. If $E_1,E_2\in \Vect (X)$ then
$E_1\otimes E_2\in \Vect (X)$.
\end{Corollary}

\begin{Proposition} \label{Tannaka}
Let $X$ be a complete connected reduced $k$-scheme. Then $\Vect
(X)$ is a rigid $k$-linear abelian tensor category.
\end{Proposition}

\begin{proof}
By the above corollary $\Vect (X)$ is a tensor category. To check
that it is abelian, it is sufficient to check that for any
homomorphism $\varphi : E_1\to E_2$ between objects $E_1$ and
$E_2$ of $\Vect (X)$ its kernel and cokernel is still in the same
category. Restricting to curves it is easy to see that $\ker
\varphi$, $\im \varphi $ and $\coker \varphi$ are locally free
(see, e.g., \cite[proof of Lemma 3.6]{No}). Since quotients of nef
bundles are nef and since we have surjections $E_1\to \im \varphi$
and $E_2^*\to (\im \varphi )^*$, $\im \varphi$ is numerically
flat. This  implies that $\ker \varphi$ and $\coker \varphi$ are
also numerically flat.
\end{proof}

For definitions and basic properties of rigid, tensor and
Tannakian categories we refer the reader to \cite{DM}.

\section{Fundamental groups in positive characteristic}

In this section we generalize the notion of S-fundamental group
scheme, defined in the curve case by Biswas, Parameswaran and
Subramanian in \cite[Section 5]{BPS}, and we compare it with other
known fundamental group schemes.

\medskip

Let $X$ be a complete connected reduced $k$-scheme and let $x\in X
$ be a $k$-point. Let  us define the fibre functor $T_x: \Vect
(X)\to k-\mod$ by sending $E$ to its fibre $E(x)$. Then $(\Vect
(X), \otimes ,T_x, \O_X)$ is a neutral Tannaka category (see
Proposition \ref{Tannaka}). Therefore by \cite[Theorem 2.11]{DM}
the following definition makes sense:

\begin{Definition} \label{Def}
The affine $k$-group scheme  Tannaka dual to this neutral Tannaka
category is denoted by $\pi^S_1(X,x)$ and it is called the
\emph{S-fundamental group scheme} of $X$ with base point $x$.
\end{Definition}

By definition, there exists an equivalence of categories $\Vect
(X)\to \pi^S_1(X,x)-\mod$ such that $T_x$ becomes a forgetful
functor for $\pi^S_1(X,x)$-modules. Inverse of this equivalence
defines  a principal $\pi^S_1(X,x)$-bundle $\ti X^S\to X$, called
the \emph{S-universal covering}, which to a $\pi^S_1(X,x)$-module
associates a numerically flat vector bundle.

Let $\pi _1^N(X,x)$ and $ \pi_1^{Et}(X,x)$ denote Nori and \'etale
fundamental group schemes, respectively. Using \cite[Proposition
2.21 (a)]{DM} it is easy to see that the following lemma holds:

\begin{Lemma}
There exist natural faithfully flat homomorphisms $\pi^S_1(X,x)\to
\pi _1^N(X,x)\to \pi_1^{Et}(X,x)$ of affine group schemes.
\end{Lemma}

Since on curves there exist strongly stable vector bundles of
degree zero and rank $r>1$ (such vector bundles are numerically
flat but not essentially finite), $\pi^S_1(X,x)\to \pi _1^N(X,x)$
is usually not an isomorphism. In fact, already non-torsion line
bundles of degree $0$ show that the S-fundamental group scheme is
usually much larger that Nori or \'etale fundamental group
schemes.

\medskip

By definition and \cite[Corollary 2.7]{DM} $\pi^S_1(X,x)$ is
isomorphic to the inverse limit of the directed system of
representations $\rho: \pi^S_1(X,x)\to G$ in algebraic $k$-group
schemes $G$, such that the image of $\rho$ is Zariski dense in
$G$. If we restrict to representations of $\pi^S_1(X,x)$ in finite
group schemes or in \'etale finite group schemes then we obtain
$\pi _1^N(X,x)$ and $ \pi_1^{Et}(X,x)$, respectively. We can
summarize this using the following obvious lemma. The formulation
for the \'etale fundamental group is left to the reader.

\begin{Lemma} \label{universal}
$\pi _1^N(X,x)$ is characterized by the following universal property:
for any representation $\rho: \pi^S_1(X,x)\to G$ in a finite $k$-group
scheme $G$, there is a unique extension to $\overline{\rho}: \pi^N_1(X,x)\to G$
such that the diagram
$$ \xymatrix{ & \pi^S_1(X,x)\ar[r]^-{\rho} \ar[d] & G\\  &\pi^N_1(X,x)\ar[ur]_-{\overline{\rho}}
&\\}$$
is commutative.
\end{Lemma}

\medskip

In \cite{dS} dos Santos used \cite{Gi} to introduce another
fundamental group scheme, which we denote by $\pi _1^F(X,x)$. It
is defined as the group scheme Tannaka dual to the Tannakian
category of $\O_X$-coherent ${\cal D}_X$-modules (corresponding to
the so called flat or stratified bundles; see \cite{Gi}).

Let us recall that there exist $\O_X$-coherent ${\cal
D}_X$-modules $(E,\nabla )$ for which $E$ is not semistable (see
\cite[proof of Theorem 1]{Gi2}). Similarly, not every numerically
flat bundle descends infinitely many times under the Frobenius
morphism. Therefore, in general, we cannot expect any natural
homomorphism between $\pi _1^S(X,x)$ and $\pi _1^F(X,x)$. But as
expected from the complex case (see Corollary
\ref{complex-surjection}), if $\mu _{\max} (\Omega_X)<0$ then the
S-fundamental group scheme carries all the algebraic information
about the fundamental group. So in this case we can obtain $\pi
_1^F(X,x)$ from this group scheme:

\begin{Proposition}
Let $X$ be a smooth projective $k$-variety. If $\mu _{\max}
(\Omega_X)<0$ then there exist a natural faithfully flat
homomorphism $\pi^S_1(X,x)\to \pi^F_1(X,x)$.
\end{Proposition}

\begin{proof}
We will need the following lemma:
\begin{Lemma}
If $\mu _{\max} (\Omega_X)<0$ then for any semistable locally free
sheaf $E$ of degree zero the canonical map $H^0(X,E)\to H^0(X,
F^*E)$ is an isomorphism.
\end{Lemma}

\begin{proof} To prove the lemma
we use the exact sequence
$$0\to \O_X\to F_*\O_X\to F_*\Omega_X.$$
Tensoring it with $E$ and taking sections we get
$$0\to H^0(X,E)\to H^0(X,F_*(F^*E))\to H^0(X, F_*( F^*E\otimes \Omega_X)).$$
Note that
$$H^0(X, F_*(F^*E\otimes \Omega_X))=H^0(X,
F^*E\otimes \Omega_X)=\Hom (F^*(E^*), \Omega_X ).$$ Now let us
recall that if $\mu _{\max} (\Omega_X)<0$ then a semistable sheaf
is strongly semistable (this fact is due to Mehta and Ramanathan;
see \cite[Theorem 2.9]{La3}). In particular, $F^*(E^*)$ is
semistable of degree larger than $\mu _{\max} (\Omega_X)$.
Therefore there are no nontrivial $\O_X$-homomorphisms between
$F^*(E^*)$ and $ \Omega_X$. Then the assertion follows from
equality $H^0(X,F_*(F^*E))=H^0(X,F^*E)$.
\end{proof}

Now we can begin the proof of the proposition. Let us recall that
a \emph{flat bundle} $\{ E_i, \sigma _i\}$ (which is equivalent to
an $\O_X$-coherent ${\cal D}_X$-module) is a sequence of locally
free sheaves $E_i$ and $\O_X$-isomorphisms $\sigma_i:
F^*E_{i+1}\to E_i$. Since $E_i$ is semistable  for large $i$, by
the above lemma $E_0$ is also semistable. Let us define the
functor between the neutral Tannaka category of flat bundles and
numerically flat bundles by sending $\{ E_i, \sigma _i\}$ to
$E_0$. Let $\{ E_i, \sigma _i\}$ and $\{ E_i', \sigma _i'\}$ be
flat bundles. Then by the above lemma applied to the sheaf $\cHom
(E_{i+1},E_{i+1}')$ we get a canonical isomorphism
$$\Hom (E_{i+1}, E_{i+1}')\simeq \Hom (E_i, E_i')$$
for every $i\ge 0$. This shows that
$$\Hom (\{ E_i, \sigma _i\} ,\{ E_i', \sigma _i'\})=\Hom (E_0 ,E_0').$$
Therefore by \cite[Proposition 2.21 (a)]{DM} to finish the proof
it is sufficient to show that if $E'$ is a numerically flat
subbundle of a bundle $E_0$ coming from the flat bundle  $\{ E_i,
\sigma _i\}$ then there exists the flat subbundle  $\{ E_i',
\sigma _i'\}$ with $E_0'\simeq E'$. Let us recall that $E_0$ has a
canonical connection $\nabla _{can} :E_0\to E_0\otimes\Omega_X$.
Since $\Hom _{\O_X}(E',  (E_0/E')\otimes \Omega_X)=0$, as follows
from our assumption, the sheaf $E'$ is preserved by the above
connection. Hence by Cartier's theorem $E'\subset F^*E_1$ descends
under the Frobenius morphism. This way we constructed $E_1'$ and
we can proceed by induction to construct the required flat bundle.
\end{proof}

\medskip

In \cite[10.25 and Proposition 10.32]{De} and \cite[Definition
3.1.3]{Sh} Deligne and Shiho introduced a pro-unipotent completion
of the fundamental group (Shiho called this group the de Rham
fundamental group scheme but it takes care only of the unipotent
part of such a hypothetical de Rham fundamental group). Let us
call this group $\pi_1^U(X,x)$. In our case, it is defined as
Tannaka dual to the neutral Tannaka category $\cal D$ consisting
of such sheaves $E$ with an integrable connection $\nabla : E\to
E\otimes \Omega_X$, which have a filtration
$$0=E_0\subset (E_1, \nabla _1)\subset \dots \subset (E_n, \nabla _n)=(E,\nabla)$$
such that we have short exact sequences
$$0\to (E_{i-1}, \nabla_{i-1})\to (E_i, \nabla_i)\to (\O_X, d)\to 0.$$
Let us note that usually the connection is not uniquely determined
by the sheaf. For example, for any closed  $1$-form $\gamma$  the
pair $(\O_X ,d+\gamma)$ is an object of $\cal D$. Also, not every
numerically flat bundle has a filtration with quotients isomorphic
to $\O_X$ (for example, no strongly stable numerically flat bundle
of rank $r\ge 2$ has such a filtration). So, in general, we cannot
expect any natural homomorphism between $\pi_1^U(X,x)$ and $\pi
_1^S(X,x)$. However, as before,  if $\mu _{\max} (\Omega_X)<0$
then  we can obtain $\pi _1^U(X,x)$ from  the S-fundamental group
scheme (as a pro-unipotent completion, although we will not try to
prove it as it is likely to be a trivial statement; see the remark
at the end of the section).

\begin{Proposition}
Let $X$ be a smooth projective $k$-variety. If $\mu _{\max}
(\Omega_X)<0$ then there exist a natural faithfully flat
homomorphism $\pi^S_1(X,x)\to \pi^U_1(X,x)$.
\end{Proposition}

\begin{proof}
Let us construct a functor $\Phi$ from $\cal D$ to the Tannaka
category of numerically flat bundles by associating to an object
$(E,\nabla)$ of $\cal D$ the sheaf $E$. Clearly, $E$ is
numerically flat so this makes sense. Let $(E_1,\nabla _1 )$ and
$(E_2,\nabla _2)$ be objects of $\cal D$. Let us take an
$\O_X$-homomorphism $\varphi: E_1\to E_2$ and consider the diagram
$$ \xymatrix{ & E_1\ar[r]^-{\nabla_1} \ar[d]^{\varphi} &E_1 \otimes \Omega_X
\ar[d]^{{\varphi} \otimes {\id}_{\Omega_X}}\\
&E_2\ar[r]^-{\nabla_2} &E_2 \otimes \Omega_X\\ }$$
Then $({\varphi} \otimes {\id}_{\Omega_X}) \circ
\nabla_1-\nabla_2\circ \varphi \in \Hom _{\O_X}(E_1, E_2\otimes
\Omega_X).$ But $E_1,E_2$ are strongly semistable and
$\mu_{\max}(\Omega _X)<0$, so $\Hom _{\O_X}(E_1, E_2\otimes
\Omega_X)=0$. Therefore  the above diagram is commutative which
shows that the functor $\Phi$ is fully faithful.

By \cite[Proposition 2.21 (a)]{DM} to finish the proof we need to
show that if $E'$ is a numerically flat subbundle of a bundle $E$
coming from  $(E, \nabla)$ then $\nabla $ induces an integrable
connection on $E'$. Then, automatically,  $E'$ has a filtration as
in the definition of $\cal D$, so it is a subobject of
$(E,\nabla)$. Note that if $\nabla$ does not preserve $E'$ then it
induces a non-trivial $\O_X$-homomorphism $E'\to (E/E')\otimes \Omega_X$.
Again one can easily see that there are no such
homomorphisms, which proves the required assertion.
\end{proof}

Finally let us formulate the following easy lemma whose proof is left to the reader:

\begin{Lemma}
Let $X$ be a smooth projective $k$-variety. If $\mu _{\max}
(\Omega_X)<0$ then every semistable locally free sheaf $E$ of
degree zero admits at most one connection. If $E$ admits a
connection $\nabla$ then it is integrable (i.e., $\nabla ^2=0$)
and its $p$-curvature vanishes. In particular, there exists $E'$
such that $(E,\nabla) \simeq (F^*E', \nabla _{can})$.
\end{Lemma}

Let us note that if $h^1(X, \O _X)\ne 0$ then $\pi^S_1(X,x)$ is
non-trivial because $\Pic ^0(X)\ne 0$. Nevertheless, the author
does not know any examples of varieties in positive characteristic
with $\mu _{\max} (\Omega_X)<0$ and a non-trivial S-fundamental
group scheme. One can show that there is no such example in
dimension $\le 2$.

\medskip

\section{Monodromy groups}

In this section we recall a few results, mostly from \cite{BPS},
generalizing them to higher dimension. Since the proofs, using our
definitions, are either the same as in \cite{BPS} or simpler we
usually skip them.

Let $G$ be a connected reductive $k$-group and let $E_G\to X$ be a
principal $G$-bundle on $X$.

\begin{Definition} \textup{(\cite[Definition 2.2]{BS})}
$E_G$ is called \emph{numerically flat} if for every parabolic
subgroup $P\subset G$ and every character $\chi: P \to {\mathbb
G}_m $ dominant with respect to some Borel subgroup of $G$
contained in $P$, the dual line bundle $L(\chi)^*$ over $E_G/P$ is
nef.
\end{Definition}

If $X$ is a smooth projective curve then $E_G$ is numerically flat
if and only if it is a strongly semistable principal $G$-bundle of
degree zero. Note that if $G$ is semisimple then a principal
$G$-bundle has always degree zero.

\begin{Lemma}
The following conditions are equivalent:
\begin{enumerate}
\item $E_G$ is numerically flat,
\item for every representation $G\to \GL (V)$  the associated
vector bundle $E_G(V)$ is numerically flat,
\item $E_G(\fg)$, associated to $E_G$ via the adjoint representation, is numerically
flat.
\end{enumerate}
\end{Lemma}

\begin{proof}
It is sufficient to prove the lemma when $X$ is a smooth
projective curve. Then 1 implies 2 because of  \cite[Theorem
3.23]{RR}. This needs a small additional argument if $G$ is not
semisimple as the radical of $G$ is not necessarily mapped to the
centre of $\GL (V)$ (the only problem is with degree of associated
bundles but this is zero as $E_G$ is numerically flat). Obviously,
2 implies 3 and 3 is equivalent to 1 by \cite[Corollary 2.8]{La3}.
\end{proof}

\medskip
Let $E_G$ be a numerically flat principal $G$-bundle. Let $\E _G
:G-\mod \to \Vect (X)$ be the functor corresponding to  $E_G$
(see, e.g., \cite[Lemma 2.3 and Proposition 2.4]{No}). Let us set
$T_G=T_x\circ\E_G$. Then $(G-\mod, \otimes, T_G, k)$ is a neutral
Tannakian category. The affine group scheme corresponding to this
category is $\Ad (E_G) _x\simeq G$. Therefore the functor
$(G-\mod, \otimes, T_G, k)\to (\Vect (X), \otimes ,T_x, \O_X )$
defines a homomorphism of groups
$$\rho (E_G) : \pi_1^S(X,x)\to \Ad (E_G)_x.$$
The image $M$ of this homomorphism is called the \emph{monodromy
group scheme} of $E_G$. One can see that $E_G$ has a reduction of
structure group to $M$ and it is the smallest such subgroup scheme
(cf. \cite[Proposition 4.9]{BPS}).

Let us recall that a subgroup of a group is called
\emph{irreducible} if it is not contained in any proper parabolic
subgroup. By \cite[Lemma 4.13]{BPS} $E_G$ is strongly stable if
and only if the reduced monodromy group $M_{red}$ is an
irreducible subgroup of $\Ad (E_G) _x\simeq G$. It is well known
that irreducible subgroups of reductive groups are reductive, so
if $E_G$ is strongly stable then by \cite[Lemma 4.12]{BPS} for
large $m$ the monodromy group of $(F^m)^*E_G$ is a reductive group
(this is analogous to the complex case; see \cite[Proposition
8.1]{BPS}).

It follows that if $E_G$ is numerically flat then for large $m$
there exists a reduction $E_P'$ of $(F^m)^*E_G$ to a parabolic
subgroup $P\subset G$ such that the monodromy group of the
extension $E_L$ of $E_P$ to the Levi quotient $q: P\to L=P/R_u(P)$
is reduced and it is an irreducible subgroup of $L$. In fact, the
monodromy group $M$ of $E_G$ is a reduced subgroup of $P$ and
$q(M)$ is the monodromy group of $E_L$.

\section{Basic properties of the S-fundamental group scheme}

In this section we prove a few basic properties of the
S-fundamental group scheme: behavior under morphisms and field
extension (in arbitrary characteristic).

\medskip

Let $f: X\to Y$ be a $k$-morphism of complete $k$-varieties. Since
pull-backs of nef bundles are nef for a $k$-point $x\in X$ there
exists an induced map $\pi _1^S(X,x)\to \pi _1^S(Y,y),$ where
$y=f(x)$.

\medskip

\begin{Lemma} \label{morphism}
Let $f: X\to Y$ be a surjective flat morphism of complete
$k$-varieties. If $f_*\O_X=\O_Y$ then $\pi _1^S(X,x)\to \pi
_1^S(Y,y)$ is a faithfully flat surjection.
\end{Lemma}

\begin{proof}
By \cite[Proposition 2.21 (a)]{DM} we need to show that
\begin{description}
\item{(a)} the functor $\Vect (Y,y)\to \Vect (X,x)$ is fully
faithful,
\item{(b)} if $E'\subset f^*E$ is a numerically flat subbundle
for $E\in \Vect(Y)$ then $E'$ is isomorphic to pull back of a
numerically flat subbundle of $E$.
\end{description}

(a) follows immediately from the projection formula:
$${\cHom }_Y(E',E'')\simeq f_*{\cHom }_X(f^*E',f^*E'')$$
by taking sections.

To prove (b) let us set $E''=(f^*E)/E'$ and
 denote by $r, r', r''$ the ranks of $E, E',E''$
respectively and let $X_y$ be the fibre over a $k$-point $y\in Y$.
Then $E'_y=E'_{X_y}$ is a numerically flat subbundle of the
trivial bundle $(f^*E)_{X_y}\simeq \O_{X_y}^r$. But $(E'_y)^*$ is
also globally generated. Since a section of such bundle has no
zeroes $E'_y$ is trivial. Similarly, $E''_y$ is trivial. In
particular, since $E'$ is $Y$-flat and $h^0(X_y,E'_y)=r'$ does not
depend on $y\in Y$ we see that $f_*E'$ is locally free of rank
$r'$ by Grauert's theorem. In the same way we prove that $f_*E''$
is locally free of rank $r''$. Since the surjective map $f^*E\to
E''$ factors through $f^*f_*E''\to E''$ we see that $f^*f_*E''\to
E''$ is a surjective map of rank $r''$ vector bundles and hence it
is an isomorphism. Therefore $f^*f_*E'\to E'$ is also an
isomorphism. Let us remind that if the pull back of a bundle is
nef then the bundle is nef. Therefore $f_*E'$ is numerically flat.
\end{proof}

\begin{Proposition} \label{P^n}
For any $k$-point $x$ of $\PP^n_k$ we have $\pi_1^S(\PP^n_k,x)=0$.
\end{Proposition}

\begin{proof}
For $n=1$ the assertion is easy as every strongly semistable
vector bundle of degree $0$ is trivial.

Let $E$ be a stable vector bundle on $\PP^2$. Then by a standard
argument $\Hom (E,E)=k$, $\ext ^2(E,E)=hom (E,E(-3))=0$ and
$$\chi (E,E)=1-\ext^1(E,E)=r^2-\Delta (E)\le 1.$$
Therefore if $E$ has vanishing Chern classes then $r=1$ and
$E\simeq \O_{\PP^2}$. Since extensions of trivial bundles on
$\PP^2$ are trivial, by Theorem \ref{loc-free} every $E\in \Vect
(\PP ^2)$ is trivial.

It is well known that a vector bundle on $\PP^n$ splits if and
only if its restriction to some plane splits (see \cite[Chapter I,
Theorem 2.3.2]{OSS}; the proof given in \cite{OSS} works in
arbitrary characteristic). Therefore if $E\in \Vect (\PP ^n)$ then
by restriction theorem the restriction of $E$ to a plane belongs
to $\Vect (\PP ^2)$, hence it is trivial, which proves that
$\pi_1^S(\PP^n_k,x)=0$.
\end{proof}

\begin{Lemma}
Let $Y$ be a smooth complete $k$-variety and let $f:X\to Y$ be the
blow-up of $Y$ along a smooth subvariety $Z\subset Y$. Then $\pi
_1^S(X,x)\to \pi _1^S(Y,y)$ is an isomorphism.
\end{Lemma}

\begin{proof}
Let $E\in \Vect (X)$. Then by Proposition \ref{P^n} restriction of
$E$ to each fibre of $f$ is trivial. Then by \cite[Theorem 1]{Is}
(which can be easily adapted to our setting) $f_*E$ is locally
free and $E\simeq f^*f_*E$.  By \cite[Proposition 2.21 (b)]{DM}
this shows that $\pi_1^S (X,x)\to \pi_1^S(Y,y)$ is a closed
immersion. Then the proof that it is faithfully flat is an easier
version of the proof of Lemma \ref{morphism}.
\end{proof}

The above lemma strongly suggests that the S-fundamental group
scheme is a birational invariant. This would follow from the above
lemma and a version of W{\l}odarczyk's result \cite{Wl} in
positive characteristic. \footnote{{\sl Added in proof:} Very recently, A. Hogadi and V. Mehta
proved birational invariance of the S-fundamental group scheme.}.

\medskip

The proof of the following lemma was motivated by the proof of
\cite[Proposition 3.1]{MS}.

\begin{Lemma}
Let $X$ be a complete variety defined over an algebraically closed
field $k$. Let  $k'$ be an algebraically closed extension of $k$.
Let $x'$ be the $k'$-point of $X_{k'}=X\times _k\Spec k'$
corresponding to a $k$-point of $x$ of $X$. Then
$\pi_1^S(X_{k'},x') \to \pi _1^S(X,x)\times _k\Spec k'$ is
faithfully flat (in particular, it is surjective).
\end{Lemma}

\begin{proof}
Let us note that if $E$ on $X_k$ is numerically flat then
$E\otimes _kk'$ is also numerically flat. By definition it is
sufficient to check this in case of smooth projective curves. But
in case of curves this follows immediately from the fact that if
$E$ on $X_k$ is stable (semistable or strongly semistable) then
$E\otimes _kk'$ is also stable (semistable or strongly semistable,
respectively); see \cite[Corollary 1.3.8 and Corollary
1.5.11]{HL}.

Let $\cT$ be the Tannakian subcategory of $\C' =(\Vect (X_{k'}),
\otimes ,T_{x'}, \O_{X_{k'}})$ whose objects are numerically flat
vector bundles $E'$ on $X_{k'}$ such that there exists $E\in \Vect
(X_k)$ such that $E'\subset E\otimes _k k'$.

Let us set $G=\pi _1^S(X,x)$ and consider the category $\cT '$ of
finite dimensional  $k'$-representations of $G_{k'}=G \times
_k\Spec k'$. Let $G_{k'}\to \GL (V')$ be a $k'$-representation.
Then by \cite[I 3.9 and 3.10]{Ja} there exists an inclusion of
$G_{k'}$-modules $V'\subset k'[G_{k'}]^{\oplus m}=(k[G]^{\oplus
m})\otimes k'$. Therefore there exists a $k$-vector subspace
$W'\subset k[G]^{\oplus m}$ such that $V'\subset W'\otimes k'$.
But there exists a finite dimensional $G$-module $W \subset
k[G]^{\oplus m}$ containing $W'$. Let $\ti X^S_{k'}$ be  the base
change of the S-universal covering of $X$. Then the vector bundle
$E'$ associated to $V'$ via this principal $G'$-bundle is a vector
subbundle of the base change of the vector bundle $E$ associated
to $W$ via the S-universal covering of $X$.

This shows that we have a natural functor $\cT ' \to \cT$ of
neutral Tannakian categories. It is easy to see that this functor
is an equivalence of Tannakian categories. Then by
\cite[Proposition 2.21 (a)]{DM} $\cT \subset \C'$ defines the
faithfully flat homomorphism $\pi_1^S(X_{k'},x) \to
\pi_1^S(X,x)\times _k\Spec k'$.
\end{proof}

\medskip
As in \cite[Proposition 3.1]{MS} one can easily see that if
$\pi_1^S(X_{k'},x) \to \pi _1^S(X,x)\times _k\Spec k'$ is a closed
immersion then every stable strongly semistable vector bundle on
$X_{k'}$ must be defined over $k$. Since this is not true already
for stable $F$-trivial bundles  (see \cite{Pa} for an example when
$X$ is a smooth curve), the above homomorphism is usually not a
closed immersion.

\medskip

Let $X$ and $Y$ be complete $k$-varieties. Let us fix $k$-points
$x\in X$ and $y\in Y$. Then we have a natural homomorphism
$$ \pi_1^S(X\times Y, (x,y)) \to \pi_1^S(X,x)\times \pi_1^S(Y,y).$$
Using embeddings of $X\times \{y\}$ and $Y\times \{x\}$ into
$X\times Y$ and Lemma \ref{morphism} one can easily see that this
homomorphism is faithfully flat. Unfortunately, it is not clear if
this is an isomorphism \footnote{In ``On the S-fundamental group scheme II''
we prove that this map is an isomorphism.}.

Note that the question is non-trivial even at the level of
characters of S-fundamental groups. For example, it is true but a
non-trivial fact that
$${\Pic } ^0(X)\times {\Pic }^0(Y)\to {\Pic }^0(X\times Y)$$
is an isomorphism on the level of $k$-points (i.e., it is an
isomorphism of the corresponding reduced schemes). But this is not
yet sufficient to conclude that a line bundle on $X\times Y$ with
a (numerically) trivial first Chern class is of the form
$p_X^*L\otimes p_Y^*M$ for some line bundles $L$ on $X$ and $M$ on
$Y$. Here we should recall that a line bundle has vanishing first
Chern class if and only if certain tensor power of this line
bundle is algebraically equivalent to zero in $\Pic X$ (see, e.g.,
\cite[Example 19.3.3]{Fu}).

\section{Some vanishing theorems for $H^1$ and $H^2$}

In this section we prove a few basic vanishing theorems for the
cohomology groups of strongly semistable sheaves with vanishing
Chern classes.

\medskip

We assume that $X$ is a smooth $d$-dimensional projective variety
defined over an algebraically closed field $k$ and $H$ is an ample
divisor on $X$ (we consider slopes only with respect to this
divisor).

\medskip

If $E\in \Vect (X)$ then for any effective divisor $D$ we have
$H^0(X, E(-D))=0$, as $E(-D)$ is semistable with negative slope.
In this section we will find similar vanishing theorems for $H^1$
and $H^2$.

\begin{Theorem}\label{vanishing} \textup{(Vanishing theorem for $H^1$)}
Assume that $X$ has dimension $d\ge 2$. Let $E\in \Vect (X)$ and
let $D$ be any ample divisor. If $DH^{d-1}>\frac{\mu_{\max}(\Omega
_X)}{p}$ then $H^1(X, E(-D))=0$.
\end{Theorem}

\begin{proof} First let us prove the following

\begin{Lemma}\label{van-lemma}
\textup{(see \cite[2.1, Crit\`ere]{Szp})} Let $E$ be a torsion
free sheaf on $X$ such that $H^0(X,F^*E(-pD) \otimes\Omega_X )=0$
and $H^1(X, F^*E(-pD))=0$. Then $H^1(X, E(-D))=0$.
\end{Lemma}

\begin{proof}
Let $B^1_X$ be the sheaf of exact $1$-forms. By definition we have
an exact sequence
$$0\to \O_X\to F_*\O_X\to F_*B^1_X\to 0.$$
By assumptions and the projection formula we have
$$H^0(X, E(-D) \otimes F_*\Omega_X )=H^0(X,F^*E(-pD) \otimes\Omega_X )=0.$$
But $F_*B^1_X$ is a subsheaf of $F_*\Omega_X^1$, so $H^0(X, E(-D)
\otimes F_*B^1_X)=0$.

Tensoring  the above sequence with $E(-D)$ and using the
projection formula we get the following exact sequence
$$0\to E(-D)\to F_*(F^*E(-pD))\to E(-D)  \otimes F_*B^1_X\to 0.$$
Since
$$H^1(X, F_*(F^*E(-pD)))=H^1(X, F^*E(-pD))=0$$
we get $H^1(X, E(-D))=0$.
\end{proof}

\medskip

The family of all strongly $H$-semistable locally free sheaves $G$ of fixed rank
with vanishing Chern classes is bounded. Hence by Serre's
vanishing theorem there exists such $m_0$ that for all $m\ge m_0$
and all such $G$ we have $H^1(X, G(-p^mD))=0$. Let us also remark
that
$$H^0(X,G(-pD)  \otimes\Omega_X )=\Hom (G^*,\Omega _X(-pD) )=0,$$
since $G^*$ is semistable with slope $0$ and by assumption
$\mu_{\max}(\Omega _X(-pD))<0$. Therefore applying Lemma
\ref{van-lemma} to $E, F^*E, (F^2)^*E,\dots $ we easily get
vanishing of $H^1(X, E(-D))$.
\end{proof}

\medskip

\begin{Corollary} \label{vanishing3}
Let $\alpha$ be a non-negative integer such that $T_X(\alpha H)$
is globally generated. Assume that $X$ has dimension $d\ge 2$. Let
$E\in \Vect (X)$ and let $D$ be any divisor such that $D-\alpha H$
is ample. If
$$DH^{d-1}>\max\left((d+1)\alpha H^d- K_XH^{d-1}, \left(1+ \frac{1}{p}\right)\alpha H^d\right)$$
then $H^1(X, E\otimes \Omega_X(-D))=0$.
\end{Corollary}

\begin{proof}
Since $T_X(\alpha H)$ is globally generated there exists a torsion
free sheaf $K$ and an integer $N$ such that we have an exact
sequence
$$0\to \Omega_X \to \O_X (\alpha H)^N \to K\to 0.$$
In particular, we have $\mu_{\max}(\Omega _X)\le \alpha H^d$
and $\mu _{\min} (K)\ge \mu _{\min} (\O_X (\alpha H)^N)=\alpha H^d$.
Since $K$ has rank $(N-d)$ we also have
$$\mu _{\max}(K)+(N-d-1)\mu _{\min }(K)\le \deg K=N\alpha H^d-K_XH^{d-1}.$$
Hence $\mu _{\max}(K) \le (d+1) \alpha H^d -K_XH^{d-1}
<DH^{d-1}=\mu _H(E^*(D))$. Because $E^*(D)$ is semistable we have
$$H^0(X,E(-D)\otimes K)=\Hom (E^*(D), K)=0.$$
Our assumptions imply that
$$\frac{\mu_{\max}(\Omega _X)}{p}< \frac{\alpha H^d}{p}\le (D-\alpha H)H^{d-1}.$$
Therefore by  Theorem \ref{vanishing} we get vanishing of
$H^1(X, E(\alpha H-D))$. Together with the above this implies
vanishing of $H^1(X, E(-D) \otimes \Omega_X )$.
\end{proof}

\medskip

\begin{Theorem} \label{vanishing2} \textup{(Vanishing theorem for $H^2$)}
Let $\alpha$ be a non-negative integer such that $T_X(\alpha H)$
is globally generated. Assume that $X$ has dimension $d\ge 3$. Let
$E\in \Vect (X)$. Let $D$ be any divisor such that $pD-\alpha H$
is ample. If
$$DH^{d-1}>\max \left(\alpha H^d, \frac{(d+1)\alpha H^d-K_XH^{d-1}}{p}\right)$$
then $H^2(X, E(-D))=0$.
\end{Theorem}

\begin{proof} First let us prove the following

\begin{Lemma}\label{van-lemma2} \textup{(cf. \cite[Proposition 2.31]{La2})}
Let $E$ be a torsion free sheaf on $X$ such that $H^0(X, E (-D) \otimes\Omega_X )=0$,
$H^0(X, F^*E (-pD) \otimes\Omega_X^2 )=0$,
$H^1(X,F^*E(-pD)  \otimes\Omega_X )=0$ and $H^2(X, F^*E(-pD))=0$.
Then $H^2(X, E(-D))=0$.
\end{Lemma}

\begin{proof}
Let $B^1_X$ and $Z^1_X$ be the sheaves of exact and closed $1$-forms, respectively.
Then we have the following short exact sequence
$$0\to F_*B^1_X\to F_*Z^1_X\mathop{\to}^{C}\Omega_X \to 0,$$
where $C$ is the Cartier operator. Tensoring it with $E(-D)$ and using
the projection formula we get the following  short exact sequence
$$0\to E(-D)\otimes F_*B^1_X \to E(-D)\otimes F_*Z^1_X\to E(-D) \otimes\Omega_X \to 0.$$
Using  definition of $Z^1_X$ we also have
an exact sequence
$$0\to F_*Z^1_X\to F_*\Omega_X\to F_*\Omega_X^2.$$
Again tensoring it with $E(-D)$ and using
the projection formula we get the following exact sequence
$$0\to E(-D)\otimes F_*Z^1_X\to  F_*(F^*E(-pD)\otimes \Omega_X)\to
F_*(F^*E(-pD)\otimes \Omega_X^2).$$
Using this sequence we see that vanishing
of $H^0(X, F^*E (-pD) \otimes\Omega_X^2 )$ and $H^1(X,F^*E(-pD)  \otimes\Omega_X )$
implies vanishing of $H^1(E(-D)\otimes F_*Z^1_X)$.
Vanishing of this group together with vanishing of
$H^0(X,E (-D) \otimes\Omega_X )$ implies vanishing of $H^1(X,E(-D)\otimes F_*B^1_X)$.
But from the long cohomology exact sequence this, together with
vanishing of $H^2(X, F^*E(-pD))$ implies vanishing of $H^2(X,
E(-D))$.
\end{proof}

As before the family of all strongly $H$-semistable locally free
sheaves $G$ of fixed rank with vanishing Chern classes is bounded and by Serre's
vanishing theorem there exists such $m_0$ that for all $m\ge m_0$
and all such $G$ we have $H^2(X, G(-p^mD))=0$.

Since $DH^{d-1}>\alpha H^d\ge {\mu_{\max}(\Omega _X)}$ we get
vanishing of $H^0(X,G(-D)  \otimes\Omega_X )$.

Now note that $\Omega_X^2$ is a subsheaf of $\bigwedge ^2(\O _X(\alpha H)^N)=\O_X (2\alpha H)^{N\choose 2}$. This implies that $$\mu_{\max}(\Omega _X^2)\le 2\alpha H^d<pDH^{d-1}=\mu
(G^*(pD)).$$ Therefore
$$H^0(X,G(-pD)  \otimes\Omega_X^2 )=\Hom (G^*(pD), \Omega_X^2)=0.$$
By assumption we have
$$pDH^{d-1}>\max\left((d+1)\alpha H^d-K_XH^{d-1},
\left(1+ \frac{1}{p}\right)\alpha H^d \right) .$$  Therefore by
Corollary \ref{vanishing3} we also have $H^1(X, G(-pD)\otimes
\Omega_X )=0$.

Now we finish  proof of the theorem by applying Lemma
\ref{van-lemma2} to $E, F^*E, (F^2)^*E,\dots $
\end{proof}

\section{Lefschetz type theorems for the S-fundamental group scheme}

In this section we prove Lefschetz type theorems for the
S-fundamental group scheme.

\medskip

Let us recall the following example. It appeared essentially in
\cite[p.181]{Szp} and then it reappeared with the interpretation
below in \cite[Section 2]{BH}.

\begin{Example}
Let $D$ be an ample effective divisor violating the Kodaira
vanishing theorem in positive characteristic (i.e., such that
$H^1(\O _X (-D))\ne 0$). Let us recall that a non-zero element $c$
of $H^1(\O_X)$ gives rise to a non-trivial extension $E$ of $\O_X$
by $\O_X$. If the class $c$ of $H^1(\O _X)$ is in the kernel of
$H^1(\O_X )\to H^1(\O_D)$ then $E_D\simeq \O_D\oplus \O_D$. By
Serre's vanishing theorem, action of the Frobenius morphism on
elements of the kernel of $H^1(\O_X )\to H^1(\O_D)$ is nilpotent.
Therefore $(F^m)^*E\simeq \O_{X}^2$ for large $m$.

This gives an example of a non-trivial representation of $\pi_1^S
(X,x)$ which is trivial on the image of $\pi_1^S(D,x)$ (obviously
this holds already on the level of the Nori's fundamental group
scheme). In particular, $\pi_1^S (D,x)\to \pi_1^S(X,x)$ is not
surjective.

\medskip
We can also interpret the above example in the following way which
explains connection with \cite{Szp}. Let $\alpha _{p^n}$ denotes
the group scheme on $X$ defined by
$$\alpha _{p^n} (U) =\{ t\in \Gamma (U, \O _U) : t^{p^n}=0\} .$$
Then we have the following exact sequence (only in fppf topology)
$$0\to \alpha_{p^n}\to {\mathbb G}_a  \mathop{\to}^{F^n} {\mathbb G}_a \to 0,$$
where the last map is given by $t\to t^{p^n}$. Using this one can
easily see that
$$H^1_{fl}(X, \alpha _{p^n})=\ker \left(H^1(X, \O_X)\mathop{\to}^{F^n} H^1(X, \O_X) \right) .$$
But $H^1_{fl}(X, \alpha _{p^n})$ is the set of
$\alpha_{p^n}$-torsors on $X$ and each such  torsor gives an
element of Nori's fundamental group. Therefore the example says
that there exists a nontrivial  element of $H^1_{fl}(X, \alpha
_{p^n})$ whose restriction to $D$ gives a trivial
$\alpha_{p^n}$-torsor.  But we know that the action of the
Frobenius on $H^1(X, \O_X (-D))$ is nilpotent so any non-zero
element of $H^1(X, \O_X (-D))$ gives such a torsor for some $n\ge
1$.
\end{Example}

\medskip

In this section $X$ is a smooth $d$-dimensional projective variety
defined over an algebraically closed field $k$ and $H$ is an ample
divisor on $X$.

\begin{Theorem} \label{Lefschetz1}
Let $D\subset X$ be any ample smooth effective divisor. If $d\ge
2$ and
$$DH^{d-1}>\mu_{\max}(\Omega _X)$$ then $\pi_1^S
(D,x)\to \pi_1^S(X,x)$ is a faithfully flat homomorphism.
\end{Theorem}

\begin{proof}
By \cite[Proposition 2.21 (a)]{DM} we need to show that
\begin{description}
\item{(a)} the functor $\Vect (X,x)\to \Vect (D,x)$ is fully
faithful,
\item{(b)} every subbundle of degree $0$ in the restriction $E_D$
of $E\in \Vect(X)$ is isomorphic to the restriction of a subbundle
of $E$.
\end{description}

To show (a) we need to prove that for $E',E''\in \Vect(X)$ the
restriction
$${\Hom }_X(E',E'')\to {\Hom }_D(E'_D,E''_D)$$
is an isomorphism. But from the short exact sequence
$$0\to {\cHom }_X(E',E'')\otimes \O_X(-D)\to {\cHom }_X(E',E'')\to {\cHom }_D(E'_D,E''_D)\to 0$$
we see that it is sufficient to show that $H^i(X,{\cHom
}_X(E',E'')\otimes \O_X(-D))=0$ for $i=0,1$. Since ${\cHom
}_X(E',E'')\in \Vect (X)$, this follows from Theorem
\ref{vanishing} and the remark preceding it.

To prove (b) let us note that for every degree $0$ subbundle of
$E_D$ there exists a Jordan--H\"older filtration $0=E_0\subset
E_1\subset \dots \subset E_m=E_D$ and some index $j$ such that
this subbundle is equal to $E_j$. So it is sufficient to lift this
filtration to a filtration of $E$.

First we prove this for sheaves such that all quotients in any
Jordan--H\"older filtration of $E$ are strongly stable.  More
precisely, let us consider the following assertion: for all
sheaves $E\in \Vect (X)$ of rank $\le r$ and such that all
quotients in any Jordan--H\"older filtration of $E$ are strongly
stable if $0=E_0\subset E_1\subset \dots \subset E_m=E_D$ is  a
Jordan--H\"older filtration of $E_D$ then $E_i$ lifts to a
subsheaf of $E$. We prove it by induction on $r$. The case $r=1$
is obvious. So assume that we know it for $r-1$ and consider a
rank $r$ sheaf $E$ satisfying the above condition. Note that it is
sufficient to lift the first subsheaf $E_1$ to a subsheaf
$E'\subset E$ and use the induction assumption for $E/E'$.

To lift $E_1$ let us take an arbitrary Jordan--H\"older filtration
$0=E_0'\subset E_1'\subset \dots \subset E_n'=E$ of $E$.  By
Theorem \ref{loc-free} each quotient $E^j=E'_j/E'_{j-1}$ is
locally free and by Theorem \ref{bogomolov} the restriction
$E^j_D$ is strongly stable. In particular, we have $n>1$ (unless
$m=1$, in which case $E$ is the required lift). Therefore there
exists some $j_0$ such that $E_1$ is isomorphic to $E^{j_0}_D$
(every non-zero map from $E_1$ to any of the sheaves $E^j_D$ is an
isomorphism). But we already know by (a) that the restriction map
$${\Hom }_X(E^{j_0},E)\to {\Hom }_D(E_1,E_D)$$
is an isomorphism so we can lift the inclusion $E_1\subset E_D$
and it clearly lifts to an inclusion.

Now let us consider the general case. Let us choose $m$ such that
all quotients in any Jordan-H\"older filtration  of $\ti
E=(F^m_X)^*E$ are strongly stable. The restriction $\ti E_D\simeq
(F^m_D)^*(E_D) $ contains $(F^m_D)^*(E_1)$ which by the above is
isomorphic to the restriction $\ti E'_D$ of some subsheaf $\ti E'$
of $\ti E$. We claim that for every $0\le i\le m$ there exists a
subsheaf $E'_i\subset (F^{m-i}_X)^*E$ such that $\ti
E'=(F^i)^*E'_i$ and $(E'_i)_D\simeq (F^{m-i}_D)^*(E_1)$. In
particular for $i=m$ we get the subsheaf of $E$ that we were
looking for. We prove the above assertion by induction on $i$. For
$i=0$ the claim is clear as we already have $E'_0=E'$. Assume that
we constructed $E'_i$ for some $i<m$. Let us set
$E_i''=((F^{m-i}_X)^*E)/E_i'$. We only need to show that there
exists $E_{i+1}'\subset (F^{m-i-1}_X)^*E$ such that
$F^*_XE_{i+1}'\simeq E_i'$. If such a sheaf does not exist then
the $\O_X$-homomorphism $E_i'\to E_i'' \otimes\Omega_X $, induced
from the canonical connection $\nabla _{can}: (F^{m-i}_X)^*E \to
(F^{m-i}_X)^*E  \otimes\Omega_X $ coming from Cartier's descent,
is non-zero (see, e.g., \cite[Theorem 2.1]{La1}; see also
\cite[Lemma 2.3]{La1} for a similar assertion). But we have a
commutative diagram
$$ \xymatrix{ & E_i'\ar[r] \ar[d] &E_i'' \otimes\Omega_X
\ar[d]\\  &(E_i')_D\ar[r]^-{0} &(E_i'')_D \otimes\Omega_D \\
}$$ where the lower map is similarly induced from the canonical
connection and it is zero because $(E_i')_D$ descends to a
subsheaf of $(F^{m-i-1}_D)^*(E_D)$ by construction. Now using the
exact sequence $$0\to \Omega_X(-D) \to \Omega_X\to  \Omega
_X|_D\to 0$$ we see that if $E_i'\to E_i'' \otimes\Omega_X \otimes \O_D $
is zero, then $E_i'\to E_i'' \otimes\Omega_X $ induces a non-zero map
$E_i'\to  E_i'' \otimes\Omega_X (-D) $
or equivalently a non-zero map $E_i'\otimes (E_i'')^*\to \Omega
_X(-D)$ . But $E_i'$ and $E_i''$ are strongly semistable of slope
$0$, so $E_i'\otimes (E_i'')^*$ is also strongly semistable. Since
by assumption $\mu _{\max} (\Omega_X(-D))<0$ the above map is
zero, a contradiction. Therefore $(E_i')_D\to (E_i'')_D  \otimes \Omega_X |_D$
is non-zero. So using the exact sequences
$$0\to \O_D(-D) \to \Omega _X|_D\to \Omega _D \to 0$$
we see that this map lifts to a non-zero map $(E_i')_D\to
(E_i'')_D \otimes \O_D(-D)$. But there are no non-zero maps
between $(E_i')_D$ and $(E_i'')_D \otimes \O_D(-D)$ because both
sheaves are semistable and the second one has smaller slope. This
finishes the proof  the theorem.
\end{proof}

\medskip

As a corollary of the above proof of (b)  we get the
following:

\begin{Corollary} \label{bogomolov-cor2}
Let $E\in \Vect (X)$, $d\ge 2$. Let $D$ be any ample smooth
effective divisor such that $DH^{d-1}>\mu_{\max}(\Omega _X)$. If
$E$ is stable then $E_D$ is also stable.
\end{Corollary}

\medskip

\begin{Theorem} \label{Lefschetz2}
Let us assume that $d\ge 3$ and $T_X(\alpha H)$ is globally
generated for some non-negative integer $\alpha$. Let $D\subset X$
be any ample smooth effective divisor such that $D-\alpha H$ is
ample. If
$$DH^{d-1}>\max \left(p\alpha H^d, (d+1)\alpha H^d-{K_XH^{d-1}}\right)$$
then $\pi_1^S (D,x)\to \pi_1^S(X,x)$ is an isomorphism.
\end{Theorem}

\begin{proof}
It is sufficient to show that for every strongly semistable
locally free sheaf $E'$ on $D$ with $\ch _1(E')\cdot H^{d-1}=0$
and $\ch _2 (E')\cdot H^{d-2}=0$ there exists a locally free sheaf
$E$ on $X$ such that $E'\simeq E_D$. Then $E$ is also strongly
semistable and   $\pi_1^S (D,x)\to \pi_1^S(X,x)$ is a closed
immersion by \cite[Proposition 2.21 (b)]{DM}. Then the assertion
follows from the previous theorem.

Let $D_n$ denote the scheme whose topological space is $D$ and the
structure sheaf is $\O _X/I_D^n$ (so $D_n$ is just the divisor
$nD$ with a natural scheme structure induced from $X$).

\begin{Lemma}
Let $S$ be a $k$-scheme of finite type. Let $\cS$ be a bounded set
of coherent sheaves on $D$. There exists $n_0$ such that for all
$n\ge n_0$ the following holds. Let $\F$ be an $S$-flat family of
locally free sheaves on $D_{n_0}$ such that $\F |_{D\times
\{s\}}\in \cS$ for every $s\in S$. Then the set $S_n\subset S$ of
points $s\in S$ such that $\F_s$ can be extended to a locally free
sheaf on $D_n\subset X$ is closed. Moreover, for large $n$, $\F
|_{D_{n_0}\times S_{n}}$ can be extended to an $S_n$-flat family
of locally free sheaves on the formal completion of $X$ along $D$.
\end{Lemma}

\begin{proof}
Let $p:D\times S \to S$ and $q:D\times S \to D$ be the natural
projections. Let  $\cExt _p^j(E, \cdot)$ be the $j$th derived
functor of $\cHom _p(E,\cdot)=p_*\circ \cHom (E, \cdot) $ (see,
e.g., \cite[10.1.7]{HL} for definition and basic properties of
these functors). Let us set
$$\G=\cExt _p^2(\F, \F\otimes q^*\O_D(-nD)).$$
Let us take $n_0$ such that for all  $n\ge n_0$,  $\Ext _D ^i(\F
_s, \F _s\otimes \O_D(-nD))$ are for all $k$-points $s\in S$ equal
to zero for $i\le 1$ and have the same dimension for $i=2$
(existence of such $n_0$ follows, e.g., from \cite[Chapter III,
Proposition 6.9 ]{Ha2}; note that we use the fact that $D$ has
dimension $\ge 2$). Then $\G$ is locally free and it commutes with
base-change. In particular, applying the base change for the map
$s:\Spec k\to S$ mapping the point $(0)$ to $s\in S$ we get an
isomorphism
$$\G _s\simeq {\Ext }^2_D(\F _s, \F _s\otimes \O_D(-nD)).$$

Using induction, it is sufficient to prove the assertion from the
lemma for $n=n_0+1$ (then in the same way one can prove it for
$n_0+2$ and so on).

Let $\ob '(\F)\in \Ext ^2_{D\times S} (\F, \F\otimes q^*\O_D(-nD))$
be an obstruction to extend $\F$ from
$D_{n_0}\times S$ to $D_{n}\times S$.
Let $\ob (\F)$ be the image of $\ob '(\F)$ under the map
$$\Ext _{D\times S}^2(\F, \F\otimes q^*\O_D(-nD)) \to
H^0(S, \cExt _p^2(\F, \F\otimes q^*\O_D(-nD))) $$ obtained from
the global to local spectral sequence $H^i (S, \cExt ^i_p
)\Rightarrow \Ext ^{i+j}_{D\times S}$ (note that by our
assumptions the beginning of the spectral sequence degenerates and
the above map is in fact an isomorphism). Then for every $k$-point
$s\in S$ the germ $\ob (\F) _s=\ob (\F_s)\in {\Ext }^2_D(\F _s, \F
_s\otimes \O_D(-nD))$ is an obstruction to extend $\F _s$ from
$D_{n_0}$ to $D_n$. So $S_n$ is just the zero set of section $\ob
(\F)$ in $S$.
\end{proof}

Let us take a flat family $\F$ of sheaves on $D$  parameterized by
some $k$-scheme $S$ of finite type and such that it contains all
sheaves $\{(F_D^n)^* E'\}_{n}$. Let $s_n\in S$ be such that $\F
_{s_n}\simeq (F_D^n)^* E'$. Consider $\F$ as a sheaf on $X\times
S$ extending it by zero (this sheaf is no longer locally free on
$X\times S$). Taking $\F'=(F_X^{n_0}\times \id_S)^* \F$ we get a
sheaf on $X\times S$, whose restriction to $D\times S$ is
$(F_D^{n_0}\times \id _S)^* \F$. But we can consider $\F'$ as an
$S$-flat family of locally free sheaves on $D_{n_0}$ and hence we
can apply the above lemma. Note that $\F '_{s_{m}}\simeq
(F_D^{m+n_0})^*E'$ can be extended to $D_{p^{m+n_0}}$  so $s_m$
belongs to $S_{p^{m+n_0}}$. But the sequence $\dots \subset
S_{n+1}\subset S_{n}\subset \dots \subset S_{n_0}=S$ stabilizes
starting with some $n_1$: $S'=S_{n_1}=S_{n_1+1}=\dots$ of $S$. By
the above there exists $m_0$ such that for all $m\ge m_0$ we have
$s_m\in S_{n_1}=S'$. Therefore for large $m$ we can extend
$(F_D^{m})^*E'$ to a locally free sheaf $\hat{E} _m$ on the formal
completion of $X$ along $D$. By \cite[Expos\'e X, Exemple 2.2]{Gr}
the pair $(X,D)$ satisfies the effective Lefschetz condition. In
particular, there exists an open set $U\supset D$ and a locally
free sheaf $E'_m$ on $U$ such that the formal completion of $E'_m$
is isomorphic to $\hat{E} _m$. Now set $E_m=j_*E'_m$, where $j:
U\hookrightarrow X$ denotes the open embedding. This is a
reflexive sheaf on $X$ such that $(F_D^m)^*E'\simeq (E_m)_D$.
Therefore $E_m$ is strongly semistable and by Theorem
\ref{loc-free} it is also locally free.

Let us take the smallest $m\ge 0$ such that $(F_D^m)^*E'$ can be
extended to a locally free sheaf $E_m$ on $X$. We need to prove
that $m=0$. Let us assume that $m\ge 1$. Replacing $E'$ with
$(F^{m-1}_D)^*E'$ we can assume that $m=1$. Then $F_D^*E'$ extends
to a vector bundle $E_1$ on $X$ and it has the canonical
connection $\nabla _{can} : F_D^*E'\to  F_D^*E'  \otimes\Omega_D$.

Let us recall that an obstruction to existence of a connection on
a vector bundle $E$ on a smooth variety $X$ is given by the Atiyah
class $A(E)\in \Ext ^1 _X(E, E\otimes \Omega _X)=H^1(X, \End
E\otimes \Omega_X)$.

 In our case we have a sequence of maps
$$H^1(X, \End E_1 \otimes \Omega_X) \mathop{\longrightarrow} ^{\alpha_0} H^1(X, \End E_1\otimes \Omega_X|_D) \mathop{\longrightarrow}^{\beta _0} H^1(D, \End
(E_1)_D\otimes \Omega_D)$$ mapping $A(E_1)$ to
$A((E_1)_D)=A(F^*_DE')=0$. Let us set $G=\End E_1$.  Note that
$\alpha _0$ is injective if $H^1(X, G\otimes \Omega_X(-D))=0$ and
$\beta _0$ is injective if $H^1(D, G_D (-D))=0$. Since $G$ is
strongly semistable, vanishing of the first cohomology group
follows from Corollary \ref{vanishing3} and our assumptions on
$DH^{d-1}$. To get vanishing of the second group we can use the
sequence
$$0\to G(-2D)\to G(-D) \to G _D(-D)\to 0$$
from which we see that it is sufficient to prove that $H^1(X,
G(-D))=H^2(X, G(-2D))=0$. This follows from Theorem
\ref{vanishing}, Theorem \ref{vanishing2} and our assumptions on
$D$ and $H$. Therefore $A(E_1)=0$ and $E_1$ has some connection
$\nabla ^1$.

We need to show that $E_1$ has a connection $\nabla$ such that on
$D$ it induces the connection $\nabla _{can}$ of $F^*_DE'_D$. Let
$\nabla ^1_D$ denotes the connection induced from $\nabla _1$ on
$D$. As above we have a sequence of maps
$$H^0(X, G \otimes \Omega_X) \mathop{\longrightarrow} ^{\alpha_1} H^0(X, G\otimes \Omega_X|_D) \mathop{\longrightarrow}^{\beta _1} H^0(D, G_D\otimes
\Omega_D).$$ Since $H^0(X, G\otimes \Omega _X(-D))=H^1(X,G\otimes
\Omega _X(-D))=0$, $\alpha _1$ is an isomorphism. Similarly,
$\beta _1$ is an isomorphism since  $H^0(D, G_D(-D))=H^1(D,
G_D(-D))=0$. Therefore $\nabla_{can}-\nabla^1_D \in H^0(D,
G_D\otimes \Omega_D)$ lifts to a unique class $\gamma \in H^0(X, G
\otimes \Omega_X)$. Then $\nabla = \nabla ^1+\gamma$ is the
required connection of $E_1$.

Again we have a sequence of maps
$$H^0(X, G \otimes  F_X^*\Omega_X) \mathop{\longrightarrow} ^{\alpha_2}
H^0(D, G_D\otimes F_D^*(\Omega_X|_D))
\mathop{\longrightarrow}^{\beta _2}H^0(D, G_D\otimes
F_D^*\Omega_D)$$ mapping the $p$-curvature of $\nabla$ to the
$p$-curvature of $\nabla _{can}$ which is $0$.

Let us recall that by assumption $\Omega _X\hookrightarrow \O_X(\alpha H)^N$ for some integer $N$.
Therefore $G\otimes (F_X^*\Omega _X)(-D)\hookrightarrow G(p\alpha H-D)^N$
and since $(p\alpha H-D)H^{d-1}<0$ we have vanishing of $H^0(X, G \otimes  (F_X^*\Omega_X )(-D))$.
Since $F_D^*(\Omega_X|_D))=(F^*_X\Omega _X)_D$ this implies that the map $\alpha_2$ is injective.
Since
$$H^0(D, G\otimes F_D^*(\O_D(-D)))=H^0(D, G(-pD))=0,$$ the map
$\beta _2$ is injective.  This proves that the $p$-curvature of
$\nabla$ is equal to $0$ and hence by Cartier's descent there
exists a sheaf $E$ on $X$ such that $E_1=F_X^*E$ and $E_D\simeq
E'$. This contradicts our assumption.
\end{proof}

\medskip

\begin{Remark}
Let us note that we do not really need Theorem \ref{Lefschetz1}
in the proof of Theorem \ref{Lefschetz2}. First as above
we prove that for any $E'\in \Vect (D)$ there
exists $E\in \Vect (X)$ such that $E_D\simeq E'$. Then we can go
back to the proof of Theorem \ref{Lefschetz1}. Point (a) is proven
in the same way as before but now point (b) is much easier.
Namely, let $E'\subset E_D$ be a  subbundle of degree $0$ in the
restriction $E_D$ of $E\in \Vect(X)$. Then we can lift $E'$ to
some bundle $E''\in \Vect (X)$. But by (a) the restriction map
$${\Hom }_X(E'',E)\to {\Hom }_D(E',E_D)$$
is an isomorphism, so inclusion $E'\subset E_D$ can be lifted to
an inclusion $E''\subset E$, which finishes the proof of (b).
\end{Remark}
\medskip

The following corollary strengthens \cite[Theorem 1.1]{BH}. Note
that in their paper Biswas and Holla used Grothendieck's Lefschetz
theorem to prove this theorem.
In our case the corollary follows immediately from Theorems \ref{Lefschetz1}
and \ref{Lefschetz2} and the universal property of the fundamental group schemes
(see Lemma \ref{universal}).

\begin{Corollary} \textup{(Lefschetz theorem for Nori's and \'etale fundamental groups)}
Let $X$ be a smooth $d$-dimensional projective variety
defined over an algebraically closed field $k$ and let $H$ be an ample
divisor on $X$. Let $D\subset X$ be any ample smooth effective divisor.
\begin{enumerate}
\item Let us assume that  $d\ge 2$ and
$$DH^{d-1}>\mu_{\max}(\Omega _X).$$
Then $\pi_1^N (D,x)\to \pi_1^N(X,x)$ and $\pi_1^{Et} (D,x)\to
\pi_1^{Et} (X,x)$ are faithfully flat.
\item
Let us assume that $d\ge 3$ and $T_X(\alpha H)$ is globally
generated for some non-negative integer $\alpha$.  Let us also
assume that $D-\alpha H$ is ample and
$$DH^{d-1}>\max \left(p\alpha H^d, (d+1)\alpha H^d-{K_XH^{d-1}}\right) .$$
Then $\pi_1^N (D,x)\to \pi_1^N(X,x)$ and $\pi_1^{Et} (D,x)\to
\pi_1^{Et} (X,x)$ are isomorphisms.
\end{enumerate}
\end{Corollary}

\medskip
In case of the local fundamental group of Nori, the
Grothendieck--Lefschetz type theorem was also proved in \cite{Me},
but without the precise bounds on the degrees of the
hypersurfaces.
\medskip

\begin{Corollary}\label{homogeneous}
Let $G$ be a reduced, connected linear algebraic group and let $X$
be a projective homogeneous $G$-space such that the
scheme-theoretic stabilizers of the action of $G$ on $X$ are
reduced. Assume that $X$ has dimension $\ge 3$. Then for any
smooth ample effective divisor  $D\subset X$ and any $k$-point
$x\in D$ the group   $\pi_1^{S} (D,x)$ is trivial.
In particular, if $D$ is a smooth hypersurface in $\PP
^d$, $d\ge 3$ then $\pi _1^S (D,x)=0$.
\end{Corollary}

\begin{proof}
We can take $\alpha =0$ in the above theorem so that we get an isomorphism
$\pi_1^{S} (D,x)\simeq \pi_1^S (X,x)$. But by \cite[Theorem 1]{MN} the S-fundamental
group scheme of $X$ is trivial, which proves the first assertion.
The last assertion also follows from Proposition \ref{P^n}.
\end{proof}

\section{Lefschetz type theorems in presence of lifting modulo $p^2$
and in characteristic zero} \label{Section-p^2}

We fix the following notation. Let $X$ be a smooth
$d$-dimensional complete variety defined over a perfect field $k$
of characteristic $p>0$. We assume throughout that $X$ has a lifting
to $W_2(k)$. Under this assumption Deligne and Illusie (and Raynaud)
showed in \cite{DI} that the Kodaira vanishing theorem is still valid
in positive characteristic. We can use their method to give stronger
Lefschetz type theorems for varieties with lifting modulo $p^2$.

\medskip

Let us recall the following lemma which is a small variation of \cite[Lemma 2.9]{DI}
(to simplify exposition we avoid the log version):

\begin{Lemma}
For any locally free sheaf $E$ and an integer $l<p$ we have
$$\sum_{i+j=l}h^j(X, E\otimes \Omega^i_X)\le \sum_{i+j=l}h^j(X, F^*E\otimes \Omega^i_X).$$
\end{Lemma}

The above lemma allows to obtain, in presence of lifting, strong
vanishing theorems for numerically flat bundles:

\begin{Corollary} \label{vvan}
For any ample divisor $D$ and any $E\in \Vect (X)$ we have
$$H^j(X,  E(-D)\otimes \Omega^i_X)=0$$
if $i+j<\min (p, d).$
\end{Corollary}

\begin{proof}
Let us note that since the family $\{(F^l)^*E\}$ is bounded we
have for  large $l$
$$H^j(X, (F^l)^*E(-p^{l}D)\otimes \Omega^i_X)=0.$$
Therefore the assertion follows by induction from the lemma
applied to the sheaves $(F^{l-1})^*E(-p^{l-1}D)$,
$(F^{l-2})^*E(-p^{l-2}D)$, $\dots, E(-D)$.
\end{proof}

\begin{Theorem}
Let $D$ be any smooth ample effective divisor on $X$.
\begin{enumerate}
\item
If $d\ge 2$ then  $\pi_1^{S} (D,x)\to \pi_1^{S} (X,x)$ is
faithfully flat.
\item
If $d\ge 3$ and $p\ge 3$ then $\pi_1^S (D,x)\to \pi_1^S(X,x)$ is
an isomorphism.
\end{enumerate}
\end{Theorem}

\begin{proof}
Using the above corollary one can follow the proofs of Theorems
\ref{Lefschetz1} and \ref{Lefschetz2} without changes (except for
the fact that vanishing of cohomology groups is much simpler).
\end{proof}

\medskip

Clearly, we get the same result also for Nori and \'etale
fundamental groups.

\bigskip

Now let $X$ be a complex projective manifold. Using Lefschetz theorems for the topological
fundamental group and the universal property of S-fundamental group scheme we get the following
theorem:

\begin{Theorem}
Let $D$ be any smooth ample effective divisor on $X$.
\begin{enumerate}
\item
If $d\ge 2$ then  $\pi_1^{S} (D,x)\to \pi_1^{S} (X,x)$ is
faithfully flat.
\item
If $d\ge 3$ then $\pi_1^S (D,x)\to \pi_1^S(X,x)$ is
an isomorphism.
\end{enumerate}
\end{Theorem}

Let us note that a similar theorem holds also for the universal
complex pro-algebraic group $\pi_1^a (X,x)$. We sketch now an
algebraic proof (in 2 we assume that $d\ge 4$).

\begin{proof}
Manivel's vanishing theorem (see \cite[Theorem A]{Ma}) implies
that for any ample divisor $D$ and any $E\in \Vect (X)$ we have
$$H^j(X, E(-D)\otimes \Omega^i_X)=0$$
if $i+j<d$ (note that the proof by reducing to characteristic $p$
and using Corollary \ref{vvan} does not quite work as we do not
know if the reduction of $E$ modulo $p$ is still in $\Vect (X)$
for some $p$). Therefore we can also give an algebraic proof of
the above Lefschetz type theorem following the proofs of Theorems
\ref{Lefschetz1} and \ref{Lefschetz2} (replacing the Frobenius
morphism with identity). In this case, in proof of Theorem
\ref{Lefschetz2}, we cannot use the Frobenius morphism to extend
$E_D$ from the divisor $D$ to $X$. But by the above vanishing
theorem we have
$$H^2(D, \End E_D \otimes \O_D(-iD))=0$$
for $i>0$. This allows us to extend $E_D$ to a vector bundle on the formal completion of $X$ along $D$
and then we can go back to the proof.
\end{proof}

Note that the above proof works only if $d\ge 4$ (as with
Grothendieck's proof of Lefschetz theorem for the Picard group).
If $d=3$ then, as one can see using Serre's duality, the above
obstruction space is never equal to zero for large $i$.
Nevertheless, in positive characteristic we could go around this
problem.

\bigskip

{\bf Acknowledgements.}

The author was partially supported by a Polish KBN grant (contract
number NN201265333).

\nocite{BaK}

\bibliography{bibl}{}
\bibliographystyle{plain}

\end{document}